\documentclass{amsart}

\usepackage{microtype}
\usepackage{fullpage}
\usepackage{bbm}
\usepackage{amsmath}
\usepackage{amssymb}
\usepackage{mathtools}
\usepackage{amsthm}
\usepackage{color}
\usepackage{soul}
\usepackage{float}
\usepackage{tikz}
\usepackage{xfrac}
\usepackage{tikz-cd}
\usetikzlibrary{trees}
\usetikzlibrary[shapes]
\usetikzlibrary[arrows]
\usetikzlibrary{patterns}
\usetikzlibrary{fadings}
\usetikzlibrary{backgrounds}
\usetikzlibrary{decorations.pathreplacing}
\usetikzlibrary{decorations.pathmorphing}
\usetikzlibrary{positioning}

\newtheorem{ThmAlpha}{Theorem}

\newtheorem{thm}{Theorem}[section]
\newtheorem{theorem}[thm]{Theorem}
\newtheorem{corollary}[thm]{Corollary}
\newtheorem{proposition}[thm]{Proposition}
\newtheorem{lemma}[thm]{Lemma}

\newtheorem*{theorem*}{Theorem}

\newtheorem*{conjecture*}{Conjecture}

\theoremstyle{definition}
\newtheorem{definition}[thm]{Definition}
\newtheorem{example}[thm]{Example}

\newtheorem{construction}[thm]{Construction}

\theoremstyle{remark}
\newtheorem{remark}[thm]{Remark}
 
\newtheorem*{convention*}{Convention}

\newcommand{\SpecR}{\mathrm{Spec}(R)}
\newcommand{\SpecZ}{\mathrm{Spec}(\mathbb{Z})}
\newcommand{\SpcK}{\mathrm{Spc}(\mathcal{K})}
\newcommand{\SpcKinv}{\mathrm{Spc}(\mathcal{K})^\vee}
\newcommand{\SpcDZinv}{\mathrm{Spc}(\mathsf{D}^{\mathrm{perf}}(\mathbb{Z}))^\vee}
\newcommand{\SpcEpinv}{\mathrm{Spc}(\mathsf{D}(\mathcal{E}_p)^\omega)^\vee}
\newcommand{\SpcEp}{\mathrm{Spc}(\mathsf{D}(\mathcal{E}_p)^\omega)}
\newcommand{\DZ}{\mathsf{D}^{\mathrm{perf}}(\mathbb{Z})}
\newcommand{\DZp}{\mathsf{D}_{\{p\}}^{\mathrm{perf}}(\mathbb{Z})}
\newcommand{\DQ}{\mathsf{D}^{\mathrm{perf}}(\mathbb{Q})}
\newcommand{\Ep}{\mathsf{D}(\mathcal{E}_p)^\omega}
\newcommand{\SHCp}{\mathcal{SH}_{(p)}^\omega}

\newcommand{\Thick}{\mathrm{thick}^{\otimes}}
\newcommand{\ThickK}{\mathrm{Th}^{\sqrt{\otimes}}(\mathcal{K})}
\newcommand{\Ridl}{\mathrm{RIdl}(R)}
\newcommand{\AnnR}{\mathrm{Ann}(R)}
\newcommand{\RoSpecR}{\Omega^{\mathrm{reg}}(\SpecR)}
\newcommand{\K}{\mathcal{K}}
\newcommand{\I}{\mathcal{I}}
\newcommand{\J}{\mathcal{J}}
\renewcommand{\P}{\mathcal{P}}
\newcommand{\E}{\mathcal{E}}
\newcommand{\C}{\mathcal{C}}
\newcommand{\D}{\mathcal{D}}
\newcommand{\A}{\mathcal{A}}
\newcommand{\oo}{\mathcal{O}}

\newcommand{\colim}{\mathop{\mathrm{colim}}}
\newcommand{\supp}{\mathop{\mathrm{supp}}}

\usepackage{hyperref}
\usepackage[nameinlink,capitalise,noabbrev]{cleveref}

\usepackage[sorting=nyt,citestyle=alphabetic,bibstyle=alphabetic,maxalphanames=5,maxbibnames=5,doi=false,isbn=false,url=false]{biblatex}
\title{De Morgan's Laws in Tensor-Triangular Geometry}
\author{Mark Lyttle}
\date{August 26, 2026.}

\begin{document}

\begin{abstract}
    There is an isomorphism between the frame of radical thick tensor ideals of an essentially small tt-category and the frame of open sets of the Hochster dual of its Balmer spectrum. This isomorphism allows us to compare topological properties of the latter with tt-geometric properties of the former. We use this to give an intrinsic definition of those tt-categories which satisfy a variant of the second De Morgan law, and furthermore give a construction for producing such tt-categories.
\end{abstract}

\maketitle

\section{Introduction}

Given a topological space $X$, there are many properties which $X$ may satisfy, for example, $X$ may be Noetherian, Hausdorff, and/or connected. In some cases, these properties are determined by the underlying frame of open sets of $X$, denoted by $\Omega(X)$, where the join of open sets is given by their union and the meet given by the interior of their intersection.

When considering the (spatial) frame $\Omega(X)$, we would like to be able to recover the original space $X$. To do this, we require $X$ to be a sober space, as we can then make use of Stone Duality (see Section \ref{SectionTopologicalDeMorganlaws}). In this case, $\Omega(X)$ often has additional meaning. The most prominent example of this is the frame of radical ideals of a commutative ring $R$, $\Ridl$. This is a spatial frame, and the sober space obtained by Stone duality is the Zariski spectrum of $R$, $\SpecR$. We therefore have an isomorphism
\begin{equation*}
    \Ridl\cong\Omega(\SpecR).
\end{equation*}
This isomorphism enables one to pass from topological properties of the space $\SpecR$ to algebraic properties of the ring $R$:

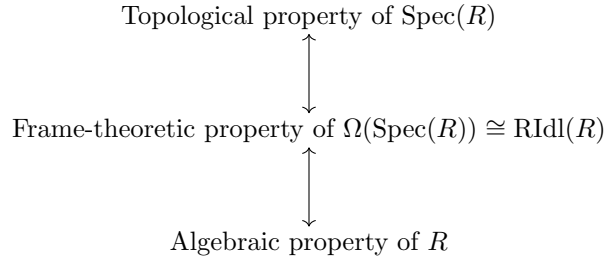
\begin{center}
    \begin{tikzpicture}[yscale=0.75]
\node at (0,2) {Topological property of $\SpecR$};
\node at (0,0) {Frame-theoretic property of $\Omega(\SpecR)\cong\Ridl$};
\node at (0,-2) {Algebraic property of $R$};
\draw [<->]  (0,1.7) -- (0,0.3);
\draw [<->]  (0,-0.3) -- (0,-1.7);
\end{tikzpicture}
\end{center}

There are many well known examples of this pipeline in use, for example, a commutative ring $R$ splits as a direct sum of subrings if and only if its spectrum splits into clopen pieces. Some less obvious examples come from \cite{ST10}, where the authors specify a certain topological property and ask whether or not the class of commutative rings whose Zariski spectrum satisfies this property is first order-axiomatizable. Their results are summarised concisely in \cite[Section 11]{ST10}.

In this paper we will focus on the case addressed in \cite{NR85}. There, the authors investigate those commutative rings whose Zariski spectra are extremally disconnected (the closure of every open subset is open)\footnote{Note that we do not assume our spaces are necessarily Hausdorff.}. They prove that in the case of $R$ being a reduced ring, $\SpecR$ is extremally disconnected if and only if the ideals of $R$ satisfy an algebraic analogue of the second De Morgan law. Niefield and Rosenthal in fact give an intrinsic definition of those reduced commutative rings which satisfy these equivalent conditions. They are the commutative Baer rings, which will be looked at further in Section \ref{SectionBaerRings}.

A tensor-triangulated category $\K$ can be thought of as a categorification of a commutative ring. When $\K$ is essentially small, its frame of radical thick tensor ideals $\ThickK$ is again a spatial frame. Stone duality therefore gives us a sober space whose frame of open sets is isomorphic to $\ThickK$, which is in fact the Hochster dual of the Balmer spectrum of $\K$, $\SpcKinv$ \cite{KP17}. With this in hand, we can establish a tt-geometric pipeline for passing from topological properties of the Balmer spectrum to intrinsic properties of the tt-category, much like the commutative algebra case:

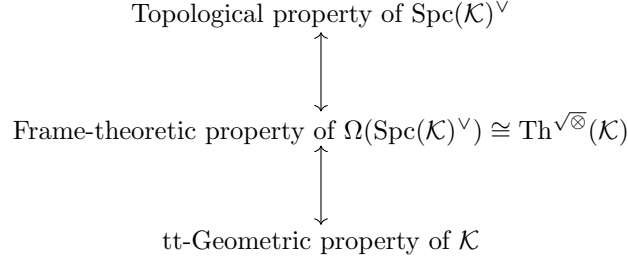
\begin{center}
    \begin{tikzpicture}[yscale=0.75]
\node at (0,2) {Topological property of $\SpcKinv$};
\node at (0,0) {Frame-theoretic property of $\Omega(\SpcKinv)\cong\ThickK$};
\node at (0,-2) {tt-Geometric property of $\K$};
\draw [<->]  (0,1.7) -- (0,0.3);
\draw [<->]  (0,-0.3) -- (0,-1.7);
\end{tikzpicture}
\end{center}

We aim to use the above framework to establish De Morgan laws for tt-categories, and in particular to give an intrinsic definition of those essentially small rigid tt-categories which satisfy the second De Morgan law, and thus have $\SpcKinv$ extremally disconnected.

\subsection*{Main Results}

We begin Section \ref{SectionttDML} by deriving tt-De Morgan laws for an arbitrary essentially small tt-category $\K$, where the pseudocomplement of a tt-ideal $\I$ is given by $\I^*=\{k\in\K~|~k\otimes\I\simeq0\}$.

\begin{enumerate}
        \item $\K$ satisfies the first tt-De Morgan law if $\Thick(\mathcal{I},\mathcal{J})^* = \mathcal{I}^*\cap\mathcal{J}^*$ for all tt-ideals $\mathcal{I},\mathcal{J}\subseteq\mathcal{K}$.
        \item $\K$ satisfies the second tt-De Morgan law if $(\mathcal{I}\cap\mathcal{J})^* = \Thick(\mathcal{I}^*,\mathcal{J}^*)$ for all tt-ideals $\mathcal{I},\mathcal{J}\subseteq\mathcal{K}$.
    \end{enumerate}

We then prove that an essentially small rigid tt-category satisfies the second tt-De Morgan Law if and only if it is a Baer tt-category, which we define in an analogous way to commutative Baer rings (see Definition \ref{DefBaertt}).

\begin{ThmAlpha}
    (Corollary \ref{CorollaryBaerttED}) Let $\K$ be an essentially small rigid tt-category. Then the following are equivalent:
    \begin{enumerate}
        \item $\K$ is a Baer tt-category.
        \item $\K$ satisfies the second tt-De Morgan law.
        \item $\SpcKinv$ is extremally disconnected.
    \end{enumerate}
\end{ThmAlpha}

An important consequence of this theorem is shown in Example \ref{ExDZ} where we consider the tt-category $\DZ$. $\SpcDZinv$ is not extremally disconnected and hence, $\DZ$ is not a Baer tt-category. This demonstrates a sharp contrast to commutative algebra, where $\mathbb{Z}$ is easily shown to be a Baer ring in Example \ref{ExampleZ}. This example thus serves as a warning when working with the derived category of a commutative ring. Although many properties of a given ring will pass to its derived category, caution is needed as this does not always occur.

In Section \ref{SectionReflection}, we give a construction for producing these Baer tt-categories. This construction is functorial, and moreover produces a left adjoint to the forgetful functor from the category of Baer tt-categories to the category of essentially small rigid tt-categories.

\begin{ThmAlpha}
    (Theorem \ref{ThmttUniversal}) Let $\mathcal{K}$ be an essentially small rigid tt-category. Then there is a Baer tt-category $B(\mathcal{K})$ and an annihilator compatible functor $\beta\colon\mathcal{K}\rightarrow B(\mathcal{K})$ such that for all annihilator compatible functors $F\colon\mathcal{K}\rightarrow\mathcal{B}$ from $\mathcal{K}$ to a Baer tt-category $\mathcal{B}$, there is a unique functor $\overline{F}\colon B(\mathcal{K})\rightarrow\mathcal{B}$ such that $\overline{F}\circ\beta=F$.
\end{ThmAlpha}

We have produced a tt-geometric version of T.P. Speed's idea of Baerification from \cite{Spe73}. This is significant as it demonstrates a construction from commutative algebra which does indeed pass through to tt-geometry. This is sometimes not the case, an example being the construction of the von Neumann-regular hull of a commutative ring $R$ from \cite{AM16}. An analogous construction for tt-categories is very desirable, however the construction of \cite{AM16} requires the use of the residue fields of $R$, a concept which has thus far proved elusive in tt-geometry (see for example \cite{BKS19}).

\subsection*{Acknowledgements}

The author would like to thank their PhD advisor Scott Balchin for many helpful conversations and comments on earlier versions of this work. We also thank David Barnes and Greg Stevenson for their careful reading of earlier drafts of this paper.

\section{Background}

\subsection{Topological De Morgan laws}\label{SectionTopologicalDeMorganlaws}

De Morgan's laws in set theory are incredibly well known. Given subsets $A$ and $B$ of some set $\mathcal{U}$, the following hold:

\begin{enumerate}
    \item $(A\cup B)^c = A^c \cap B^c$.
    \item $(A\cap B)^c = A^c \cup B^c$.
\end{enumerate}

We can however state these laws more generally in the setting of frames, where the power set of $\mathcal{U}$ is an example of a frame (in fact a complete Boolean algebra, see Section \ref{SectionBaerRings}).

\begin{definition}\label{DefinitionFrame}
    Let $L$ be a complete lattice, that is, a partially ordered set such that every subset has a meet ($\wedge$) and a join ($\vee$). $L$ is a \textbf{frame} if it satisfies the following distributivity law:
    \begin{equation*}
        \Big(\bigvee A\Big)\wedge b = \bigvee\{a\wedge b~|~a\in A\}
    \end{equation*}
    for all subsets $A\subseteq L$ and for all $b\in L$. Every frame comes equipped with a \textbf{Heyting operation}, which is defined as follows:
\begin{equation*}
    (a\Rightarrow b) := \bigvee\{c\in L~|~ a\wedge c\leq b\}
\end{equation*}

for elements $a,b\in L$. From this, we can define the \textbf{pseudocomplement} of an element $a\in L$, which is given by

\begin{equation*}
    \neg a := (a\Rightarrow 0) = \bigvee\{c\in L~|~ a\wedge c= 0\}
\end{equation*}

where $0$ denotes the bottom element of the frame.
\end{definition}

For the power set of $\mathcal{U}$, $\mathbb{P}(\mathcal{U})$, the meet of a collection of subsets is given by their intersection, while the join is given by their union. The top and bottom elements of this frame are $\mathcal{U}$ and $\varnothing$ respectively, and the pseudocomplement of a subset $A$ is given by its complement $A^c=\mathcal{U}\setminus A$. One can easily prove that the above distributivity law holds for $\mathbb{P}(\mathcal{U})$. 

We now define De Morgan's laws in this frame-theoretic setting:

\begin{definition}\label{DefinitionFrameDeMorganlaws}
    Let $L$ be a frame and let $A,B\subseteq L$. Then
    \begin{enumerate}
    \item $L$ satisfies the \textbf{first De Morgan law} if $\neg(A\vee B) = \neg A \wedge \neg B$.
    \item $L$ satisfies the \textbf{second De Morgan law} if $\neg(A\wedge B) = \neg A \vee \neg B$.
\end{enumerate}
\end{definition}

If we again consider the example of $\mathbb{P}(\mathcal{U})$, the above definition recovers the usual De Morgan laws from set theory, which are known to always hold. In the more general setting of an arbitrary frame, \cite{Fre72} shows that the first De Morgan law is still always true. However, this does not apply to the second De Morgan law, and Theorem 1 of \cite{Joh79} gives a necessary and sufficient condition for a frame to satisfy it (Theorem \ref{TheoremJohnstone}).

A frame $L$ is \textbf{spatial} if for all $a,b\in L$ such that $a\nleq b$, we have a map $p\colon L\rightarrow\mathbf{2}:=\{0\leq 1\}$ with $p(a)=1$ and $p(b)=0$, while a topological space $X$ is \textbf{sober} if every non-empty irreducible closed subset of $X$ contains a unique generic point. There is an equivalence of categories between the category of spatial frames and sober spaces
\begin{equation*}
    S\colon\mathsf{Spat}\rightleftarrows\mathsf{Sob}\colon\Omega
\end{equation*}
known as Stone duality, so given a spatial frame $L$ we have a sober space $S(L)$ such that there is an isomorphism of frames
\begin{equation*}
    L\cong\Omega(S(L))
\end{equation*}
between $L$ and the frame of open sets of $S(L)$. This isomorphism will allow us to compare topological properties of the space $S(L)$ with properties of the original spatial frame in question. With this in hand, we can give Johnstone's condition for a frame to satisfy the second De Morgan law.

\begin{definition}
    Let $X$ be a topological space. Then $X$ is said to be \textbf{extremally disconnected} if the closure of every open set of $X$ is open.
\end{definition}

\begin{theorem}\label{TheoremJohnstone}
    (\cite[Theorem 1]{Joh79}) Let $L$ be a spatial frame and let $S(L)$ be the sober space obtained by Stone duality. Then $L$ satisfies the second De Morgan law if and only if $S(L)$ is extremally disconnected.
\end{theorem}

\begin{example}\label{Example3set}
    Consider the three point set $X=\{0,1,2\}$ with the excluded point topology. That is, a subset $U$ of $X$ is open if it does not contain the point $0$, or is equal to the whole space. This space can be represented in the following diagram where closure goes downwards.
    
    \begin{figure}[H]
    \centering
\begin{tikzpicture}[yscale=-1]
\draw[white, fill=white] (0,0) circle (.3cm);
\node[circle] at (0,0) (e) {};
\draw[fill = black] (0,0) circle (0.1cm);
\draw[draw = black] (-2,-1.5) circle (0.1cm);
\draw[draw = black] (2,-1.5) circle (0.1cm);
\node at (0,.5) {$0$};
\node at (-2,-2.0) {$1$};
\node at (2,-2.0) {$2$};
\draw [->,line join=round,decorate, decoration={zigzag, segment length=6, amplitude=.9,post=lineto, post length=2pt}]  (-1.8,-1.25) -- (e);
\draw [->,line join=round,decorate, decoration={zigzag, segment length=6, amplitude=.9,post=lineto, post length=2pt}]  (1.8,-1.25)  -- (e);
\end{tikzpicture}\caption{Three point set with the excluded point topology. The filled circle represents a closed point and the arrows represent specialization.}\label{Figure3set}
\end{figure}
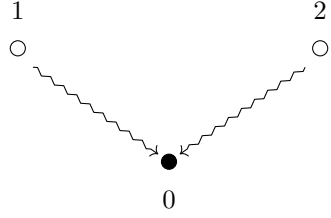

Recall that in the frame $\Omega(X)$, the meet of open sets is given by the interior of their intersection, the join by their union, and the pseudocomplement of an open set by the interior of its complement. Now consider the open sets $A=\{1\}$ and $B=\{2\}$. Then
\begin{align*}
    \neg(A\wedge B) &= \neg((\{1\}\cap\{2\})^\circ) = \neg(\varnothing) = X,\\
    \text{while } ~~~ \neg A \vee \neg B &= \{0,2\}^\circ\cup\{0,1\}^\circ = \{2\}\cup\{1\} = \{1,2\} \neq X,
\end{align*}
which shows that the second De Morgan law does not hold for the frame $\Omega(X)$. By Theorem \ref{TheoremJohnstone}, $X$ cannot be extremally disconnected, which can be verified by observing that $\overline{\{1\}}=\{0,1\}$ which is not an open set. Note that this space arises as 
\begin{equation*}
    X \cong \mathrm{Spec}\big(\sfrac{k[x,y]_{(x,y)}}{(xy)}\big).
\end{equation*}
\end{example}

\subsection{Commutative Baer Rings}\label{SectionBaerRings}

The frame of radical ideals of a commutative ring $R$, $\Ridl$, forms a spatial frame, with the meet of two radical ideals given by their intersection and their join given by the radical of their sum (the smallest radical ideal containing their union). The space obtained via Stone duality is the usual Zariski spectrum $\SpecR$. One can therefore conclude that the frame $\Ridl$ satisfies the second De Morgan law if and only if $\SpecR$ is extremally disconnected. In \cite{NR85}, the authors derive algebraic De Morgan laws, and give an intrinsic definition of those commutative (reduced) rings which satisfy the second algebraic De Morgan law. They are the commutative Baer rings, objects which have been studied as a potential algebraic analogue of von Neumann algebras.

The pseudocomplement of a radical ideal $I$ is given by the radical of its \textbf{annihilator ideal}, $I^*:=\{r\in R~|~rI=0\}$. We can now define the algebraic De Morgan laws, as derived by Niefield and Rosenthal.

\begin{definition}\label{DefinitionAlgeraicDeMorganlaws}
    (\cite[Definition 2.3.]{NR85}) Let $R$ be a commutative ring. Then $R$ satisfies:
    \begin{enumerate}
        \item The \textbf{first algebraic De Morgan law} if $(A+B)^*=A^*\cap B^*$ for all ideals $A,B\unlhd R$.
        \item The \textbf{second algebraic De Morgan law} if $(A\cap B)^*=A^*+B^*$ for all ideals $A,B\unlhd R$.
    \end{enumerate}
\end{definition}

In Section \ref{SectionTopologicalDeMorganlaws}, we commented on the fact that every frame satisfies the first De Morgan law by \cite{Fre72}. We now confirm that every commutative ring satisfies the first algebraic De Morgan law.

\begin{proposition}
    Let $R$ be a commutative ring. Then $R$ satisfies the first algebraic De Morgan law.
\end{proposition}

\begin{proof}
    $A\subseteq A+B$ and $B\subseteq A+B$, hence $(A+B)^*\subseteq A^*$ and $(A+B)^*\subseteq B^*$, and so $(A+B)^*\subseteq A^*\cap B^*$. 
    
    Conversely, let $x\in A^*\cap B^*$ and $y\in A+B$. Then $y=a+b$ for $a\in A$ and $b\in B$, so 
    \begin{equation*}
        xy=xa+xb=0
    \end{equation*} 
    since $x\in A^*\cap B^*$. Hence, $x\in(A+B)^*$ and so $A^*\cap B^*\subseteq(A+B)^*$.
\end{proof}

Every commutative ring in fact satisfies one inclusion of the second algebraic De Morgan law. To see this, let $r\in A^*+B^*$. Then there are $a\in A^*$ and $b\in B^*$ such that $r=a+b$. Now, let $c\in A\cap B$. Then 
\begin{equation*}
    rc=(a+b)c=ac+bc=0
\end{equation*}
so $r\in(A\cap B)^*$ as required. The following example will demonstrate that the reverse inclusion does not always hold.

\begin{example}\label{ExampleNot2DML}
    Consider the commutative ring $R=\mathbb{Z}[x]/(x^2-1)$ and the ideals $A=(x+1)$ and $B=(x-1)$. Then
    \begin{align*}
        A\cap B &= (x+1)\cap(x-1)=0 \implies (A\cap B)^* = R,\\
        \text{while } ~~~ A^*+B^* &= (x+1)^*+(x-1)^* = (x-1)+(x+1)\neq R
    \end{align*}
    since $1\notin(x-1)+(x+1)$. Hence, $(A\cap B)^*\nsubseteq A^*+B^*$.
\end{example}

 We now restrict our discussion to commutative \textbf{reduced} rings. In this case, the second algebraic De Morgan law holds precisely when $R$ is a commutative Baer ring.

 \begin{definition}
    Let $R$ be a commutative ring. Then $R$ is a \textbf{Baer ring} if the annihilator of every ideal of $R$ is generated by an idempotent, i.e., for every ideal $I\unlhd R$, we have $I^*=(e)$ for some idempotent element $e\in R$. 
\end{definition}

\begin{theorem}
    (\cite[Theorem 1]{NR85}) Let $R$ be a commutative reduced ring. Then $R$ satisfies the second algebraic De Morgan law if and only if $R$ is a Baer ring.
\end{theorem}

Again consider the commutative reduced ring $R=\mathbb{Z}[x]/(x^2-1)$ of Example \ref{ExampleNot2DML}. By the above theorem, since $R$ does not satisfy the second algebraic De Morgan law, $R$ cannot be a Baer ring. To see this directly, observe that the annihilator ideals $(x+1)=(x-1)^*$ and $(x-1)=(x+1)^*$ are not idempotent generated since the only idempotent elements of $R$ are $0$ and $1$.

Thus far we have established that $\Ridl$ satisfies the second (frame-theoretic) De Morgan law if and only if $\SpecR$ is extremally disconnected, and in the case of $R$ being reduced, $\Ridl$ satisfies the second algebraic De Morgan law if and only if $R$ is a Baer ring. It is no surprise that when $R$ is a commutative reduced ring, these conditions are all equivalent:

\begin{theorem}\label{TheoremBaeriffED}
    (\cite[Theorem 2]{NR85}) Let $R$ be a commutative reduced ring. Then $R$ is a Baer ring if and only if $\SpecR$ is extremally disconnected.
\end{theorem}

We do not include the proof here but rather explain the intuition behind this equivalence. Recall that a \textbf{complete Boolean algebra} $B$ is a frame in which every element has a complement, i.e., for every $b\in B$ there is an element $b'$ such that $b\vee b'=1$. When $R$ is reduced, the collection of annihilator ideals of $R$ is a complete Boolean algebra by \cite[Theorem 1.1.]{DT21}, where the complement of an annihilator ideal $I^*$ is given by its annihilator $I^{**}$. Note that for a non-reduced ring this may only produce a pseudocomplement. We denote the complete Boolean algebra of annihilator ideals of $R$ by $\AnnR$.

An open set $U$ of a topological space $X$ is said to be \textbf{regular open} if $U=\overline{U}^\circ$. The collection of regular open subsets of a space is also an example of a complete Boolean algebra, and by \cite[Proposition 1.8.]{DT21}, there is an isomorphism
\begin{equation*}
    \AnnR\cong\RoSpecR,
\end{equation*}
where $\RoSpecR$ is the complete Boolean algebra of regular open subsets of $\SpecR$. As a consequence, a commutative reduced ring $R$ being a Baer ring, i.e., every annihilator ideal being idempotent generated, is equivalent to every regular open subset of $\SpecR$ being clopen. This is in turn equivalent to $\SpecR$ being extremally disconnected, since an extremally disconnected space can only have a regular open subset if it is clopen, as in this case we have $U=\Big(\overline{U}\Big)^\circ=\overline{U}$, due to the closure of every open set being open.

The following example is included to demonstrate the equivalence of Theorem \ref{TheoremBaeriffED}.

\begin{example}\label{ExampleZ}
    Consider the ring of integers $\mathbb{Z}$ and let $I$ be a non-zero ideal of $\mathbb{Z}$. Then $I^*=0$ since $\mathbb{Z}$ is a domain. Since 0 is idempotent, $\mathbb{Z}$ is a Baer ring (for this reason, every domain is a Baer ring).
    
    Now consider the space $\SpecZ$:

    \begin{figure}[H]
    \centering
\begin{tikzpicture}[yscale=1]
\draw[white, fill=white] (0,0) circle (.3cm);
\node[circle] at (0,0) (e) {};
\draw[draw = black] (0,0) circle (0.1cm);
\draw[fill = black] (-4,-1.25) circle (0.1cm);
\draw[fill = black] (-2,-1.5) circle (0.1cm);
\draw[fill = black] (0,-1.75) circle (0.1cm);
\draw[fill = black] (2,-1.5) circle (0.1cm);
\node at (0,.5) {$(0)$};
\node at (-4,-1.75) {$(2)$};
\node at (-2,-2.0) {$(3)$};
\node at (0,-2.25) {$(5)$};
\node at (2,-2.0) {$(7)$};
\node at (4,-1.75) {$\cdots$};
\draw [->,line join=round,decorate, decoration={zigzag, segment length=6, amplitude=.9,post=lineto, post length=2pt}]  (e) -- (-3.75,-1.2);
\draw [->,line join=round,decorate, decoration={zigzag, segment length=6, amplitude=.9,post=lineto, post length=2pt}]  (e) -- (3.75,-1.2);
\draw [->,line join=round,decorate, decoration={zigzag, segment length=6, amplitude=.9,post=lineto, post length=2pt}]  (e) -- (-1.8,-1.25);
\draw [->,line join=round,decorate, decoration={zigzag, segment length=6, amplitude=.9,post=lineto, post length=2pt}]  (e)  -- (1.8,-1.25);
\draw [->,line join=round,decorate, decoration={zigzag, segment length=6, amplitude=.9,post=lineto, post length=2pt}]  (e)  -- (0.0,-1.4);
\end{tikzpicture}\caption{The Zariski spectrum of $\mathbb{Z}$.}\label{fig:SpecZ}
\end{figure}
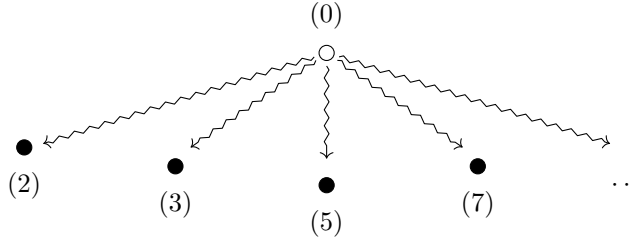

where each of the non-zero ideals represent a closed point, $(0)$ is a generic point, and the direction of closure is downwards. The closure of $(0)$ is therefore equal to $\SpecZ$, and since every open set must contain (0), the closure of any open set must be equal to $\SpecZ$ itself, which is of course open. Hence, $\SpecZ$ is extremally disconnected.
\end{example}

\begin{example}
    Returning to the example of $R=\mathbb{Z}[x]/(x^2-1)$, by Theorem \ref{TheoremBaeriffED} $\SpecR$ is not extremally disconnected. To see this, consider $\SpecR$ which is as follows:
    
    \begin{figure}[H]
    \centering
\begin{tikzpicture}[yscale=1]
\draw[fill = black] (-4,-1.5) circle (0.1cm);
\draw[draw = black] (-3,0) circle (0.1cm);
\node[circle] at (-3,0) (e) {};
\draw[fill = black] (-2,-1.5) circle (0.1cm);
\draw[fill = black] (0,-1.25) circle (0.1cm);
\draw[fill = black] (2,-1.5) circle (0.1cm);
\draw[draw = black] (3,0) circle (0.1cm);
\node[circle] at (3,0) (f) {};
\draw[fill = black] (4,-1.5) circle (0.1cm);
\node at (-6,-1.5) {$\cdots$};
\node at (-4,-2.125) {$(5,x-1)$};
\node at (-2,-2.125) {$(3,x-1)$};
\node at (-3,.5) {$(x-1)$};
\node at (0,-1.625) {$(2,x-1)=(2,x+1)$};
\node at (3,.5) {$(x+1)$};
\node at (2,-2.125) {$(3,x+1)$};
\node at (4,-2.125) {$(5,x+1)$};
\node at (6,-1.5) {$\cdots$};
\draw [->,line join=round,decorate, decoration={zigzag, segment length=6, amplitude=.9,post=lineto, post length=2pt}]  (e) -- (-5.5,-1.1);
\draw [->,line join=round,decorate, decoration={zigzag, segment length=6, amplitude=.9,post=lineto, post length=2pt}]  (e) -- (-3.9,-1.25);
\draw [->,line join=round,decorate, decoration={zigzag, segment length=6, amplitude=.9,post=lineto, post length=2pt}]  (e) -- (-2.1,-1.25);
\draw [->,line join=round,decorate, decoration={zigzag, segment length=6, amplitude=.9,post=lineto, post length=2pt}]  (e)  -- (-0.2,-1.2);
\draw [->,line join=round,decorate, decoration={zigzag, segment length=6, amplitude=.9,post=lineto, post length=2pt}]  (f) -- (5.5,-1.1);
\draw [->,line join=round,decorate, decoration={zigzag, segment length=6, amplitude=.9,post=lineto, post length=2pt}]  (f) -- (3.9,-1.25);
\draw [->,line join=round,decorate, decoration={zigzag, segment length=6, amplitude=.9,post=lineto, post length=2pt}]  (f) -- (2.1,-1.25);
\draw [->,line join=round,decorate, decoration={zigzag, segment length=6, amplitude=.9,post=lineto, post length=2pt}]  (f)  -- (0.2,-1.2);
\end{tikzpicture}\caption{The Zariski spectrum of the ring $\mathbb{Z}[x]/(x^2-1)$.}\label{fig:2xSpecZglued}
\end{figure}
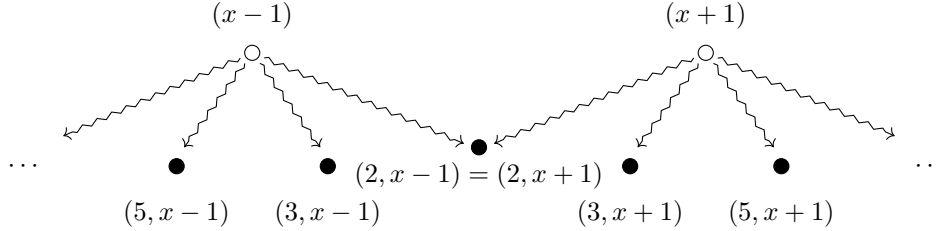

The closure of the point $(x-1)$ is the set
\begin{equation*}
    \{(x-1)\}\cup\{(p,x-1)~|~p\text{ prime}\}
\end{equation*}
which is not open (its complement is clearly not closed).

\end{example}

\subsection{Tensor-Triangular Geometry}

Tensor-triangular geometry was first introduced in \cite{Bal05} to unify various areas of mathematics, for example stable homotopy theory, algebraic geometry, and representation theory. This is done by studying arbitrary tensor-triangulated (tt-)categories, which are triangulated categories with a symmetric monoidal structure. When dealing with essentially small tt-categories, that is, those tt-categories which have only a set of objects and morphisms, the thick tensor ideals are of particular interest.

\begin{definition}
    (\cite{Bal05}, Definition 1.2.) A \textbf{thick tensor (tt-)ideal} $\I$ of an essentially small tt-category $\K$ is a full subcategory of $\K$, containing 0, such that:
    \begin{enumerate}
        \item $\I$ is \textbf{triangulated}, i.e., for any exact triangle $x\rightarrow y\rightarrow z\rightarrow\Sigma x$ in $\K$ such that two of $x,y,z$ are in $\I$, so is the third.
        \item $\I$ is \textbf{thick}, i.e., $\I$ is closed under taking direct summands (if $x\oplus y\in\I$ then $x,y\in\I$).
        \item $\I$ is an \textbf{ideal}, i.e., if $x\in\I$ and $k\in\K$ then $x\otimes k\in\I$.
    \end{enumerate}
    The tt-ideal generated by an object $x\in\K$, that is the smallest tt-ideal containing $x$, is denoted by thick$^\otimes(x)$, and the set of all tt-ideals of $\K$ is denoted by Th$^\otimes(\K)$.
\end{definition}

We now restrict our discussion to the \textbf{radical tt-ideals}, those tt-ideals $\I$ where $x^{\otimes n}\in\I$ implies that $x\in\I$. By \cite[Theorem 3.1.9.]{KP17}, the collection of all radical tt-ideals of an essentially small tt-category $\K$, $\ThickK$, forms a spatial frame. Hence, Stone duality gives us a sober space $S(\ThickK)$ where the frame of open sets of $S(\ThickK)$ is isomorphic to the frame $\ThickK$. The authors of \cite{KP17} then go on to show in Proposition 3.4.1. that $S(\ThickK)$ is homeomorphic to the Hochster dual of the usual Balmer spectrum of $\K$ (\cite[Definition 2.1.]{Bal05}). That is, $S(\ThickK)$ is homeomorphic to $\SpcK$ equipped with the inverse topology, defined by taking the Thomason subsets of $\SpcK$ as open sets, where a \textbf{Thomason set} is an arbitrary union of closed sets each with quasi-compact complement. As a result of this homeomorphism, we will henceforth denote the space $S(\ThickK)$ as $\SpcKinv$. The open sets of $\SpcKinv$ are given by the supports of tt-ideals of $\K$, that is, for $\mathcal{I}\subseteq\K$ we have $\supp(\I)=\{\mathcal{P}\in\text{Spc}(\mathcal{K})^\vee~|~\mathcal{I}\nsubseteq\mathcal{P}\}$.

The isomorphism
\begin{equation*}
    \ThickK\cong\Omega(\SpcKinv)
\end{equation*}
will therefore allow us to compare properties of a given collection of radical tt-ideals with topological properties of the Balmer spectrum. In particular, an application of Theorem \ref{TheoremJohnstone} tells us that the frame $\ThickK$ will satisfy the second De Morgan law if and only if $\SpcKinv$ is extremally disconnected. In Section \ref{SectionttDML}, we will define the notion of Baer tt-categories, thus giving an intrinsic definition of those tt-categories which satisfy these equivalent conditions.

\begin{example}\label{ExampleThomason}
    One of the main examples of tt-categories that we will consider are the derived categories of commutative rings. If we restrict to the perfect complexes within these categories, i.e., those which are quasi-isomorphic to a bounded complex of finitely generated projective modules, we get an essentially small tt-category, denoted by $\mathsf{D}^{\mathrm{perf}}(R)$ for $R$ a commutative ring. A useful result of Thomason (\cite[Theorem 3.15.]{Tho97}) allows us to easily compute the Balmer spectrum of $\mathsf{D}^{\mathrm{perf}}(R)$ as
    \begin{equation*}
        \mathrm{Spc(}\mathsf{D}^{\mathrm{perf}}(R))\cong\mathrm{Spec}(R).
    \end{equation*}
\end{example}

\section{The Second Tensor-Triangular De Morgan Law}\label{SectionttDML}

\subsection{Tensor-Triangular De Morgan Laws}

In this section, we want to classify those essentially small tt-categories whose frame of radical tt-ideals satisfies the second De Morgan law. Analogously to the algebraic De Morgan laws of Niefield and Rosenthal, the first step is to derive tensor-triangular De Morgan laws. For this, we need to define the pseudocomplement in the frame $\ThickK$.

\begin{definition}
    Let $\mathcal{K}$ be an essentially small tt-category and let $\mathcal{S}\subseteq\mathcal{K}$ be a subcategory. Then the \textbf{tt-annihilator ideal} of $\mathcal{S}$ is the tt-ideal $\mathcal{S}^*:=(\mathcal{S}\Rightarrow0)=\{k\in\mathcal{K}~|~k\otimes\mathcal{S}\simeq0\}$.
\end{definition}

\begin{lemma}\label{LemttAnnRadical}
    Let $\mathcal{K}$ be an essentially small tt-category with no nilpotent objects. Then every tt-annihilator ideal of $\mathcal{K}$ is radical.
\end{lemma}

\begin{proof}
    Let $x\in\sqrt{\mathcal{I}^*}$ for some tt-ideal $\mathcal{I}\subseteq\mathcal{K}$. Then $x^{\otimes n}\otimes\mathcal{I}\simeq0$ for some $n\in\mathbb{N}$, i.e., $x^{\otimes n}\otimes k\simeq0$ for all $k\in\mathcal{I}$, and so, $0\simeq x^{\otimes n}\otimes k^{\otimes n}\simeq(x\otimes k)^{\otimes n}$. Now, since $\mathcal{K}$ has no nilpotent objects, we have that $x\otimes k\simeq 0$ for all $k\in\mathcal{I}$. Hence, $x\in\mathcal{I}^*$, so $\I^*$ is radical.
\end{proof}

With the pseudocomplement defined in the frame of radical tt-ideals, we are now in position to derive the De Morgan laws in this setting.

Let $\Omega(\SpcKinv)$ be the frame of open sets of the Hochster dual of the Balmer spectrum of an essentially small tt-category $\K$. We can then state the frame-theoretic De Morgan laws of Definition \ref{DefinitionFrameDeMorganlaws} for this situation.
\begin{enumerate}
    \item $\neg(\supp(\mathcal{I})\vee\supp(\mathcal{J})) = \neg\supp(\mathcal{I})\wedge\neg\supp(\mathcal{J})$
    \item $\neg(\supp(\mathcal{I})\wedge\supp(\mathcal{J})) = \neg\supp(\mathcal{I})\vee\neg\supp(\mathcal{J})$
\end{enumerate}
for tt-ideals $\mathcal{I},\mathcal{J}\subseteq\mathcal{K}$. We must clarify what is meant by $\neg\supp(\mathcal{I})$.
\begin{equation*}
    \begin{split}
        \neg\supp(\mathcal{I}) &= \big(\supp(\mathcal{I})\Rightarrow\supp(0)\big)\\ &= \bigvee\{\supp(\mathcal{J})\in \Omega(\text{Spc}(\mathcal{K})^\vee)~|~ \supp(\mathcal{I}\otimes\mathcal{J})= \supp(0)\}\\ &= \supp\big(\Thick(\J\subseteq\K~|~\I\otimes\J\simeq0)\big) = \supp(\mathcal{I}^*).
    \end{split}
\end{equation*}
Returning to the first De Morgan law, we arrive at the following:
\begin{equation*}
    \begin{split}
    \neg(\supp(\mathcal{I})\vee\supp(\mathcal{J})) &= \neg\supp(\mathcal{I})\wedge\neg\supp(\mathcal{J})\\
    \Longleftrightarrow~~~\neg(\supp(\Thick(\I,\J))) &= \supp(\mathcal{I}^*)\wedge\supp(\mathcal{J}^*)\\
    \Longleftrightarrow~~~\supp(\Thick(\I,\J)^*) &= \supp(\mathcal{I}^*\cap\mathcal{J}^*)\\
    \Longleftrightarrow~~~\Thick(\I,\J)^* &= \mathcal{I}^*\cap\mathcal{J}^*.
    \end{split}
\end{equation*}
The second tt-De Morgan law is derived similarly.

\begin{definition}\label{DefinitionttDeMorganlaws}
    Let $\mathcal{K}$ be an essentially small tt-category. Then $\mathcal{K}$ satisfies:
    \begin{enumerate}
        \item The \textbf{first tt-De Morgan law} if $\Thick(\I,\J)^* = \mathcal{I}^*\cap\mathcal{J}^*$ for all tt-ideals $\mathcal{I},\mathcal{J}\subseteq\mathcal{K}$.
        \item The \textbf{second tt-De Morgan law} if $(\mathcal{I}\cap\mathcal{J})^* = \Thick(\I^*,\J^*)$ for all tt-ideals $\mathcal{I},\mathcal{J}\subseteq\mathcal{K}$.
    \end{enumerate}
\end{definition}

In analogy with the ring-theoretic case, we expect that every essentially small tt-category satisfies the first tt-De Morgan law, and one inclusion of the second.

\begin{proposition}\label{Prop1ttDML}
    Let $\mathcal{K}$ be an essentially small tt-category. Then $\mathcal{K}$ satisfies the first tt-De Morgan law.
\end{proposition}

\begin{proof}
    Since $\I,\J\subseteq\Thick(I,J)$, we have $\Thick(\I,\J)^*\subseteq\I^*,\J^*$. Hence, $\Thick(\I,\J)^*\subseteq\I^*\cap\J^*$.
    
    Conversely, let $x\in\mathcal{I}^*\cap\mathcal{J}^*$. Then $\I\subseteq(x)^*$ and $\J\subseteq(x)^*$. Hence, $\Thick(\I,\J)\subseteq(x)^*$ which implies that $x\in(x)\subseteq\Thick(\I,\J)^*$, and thus $\mathcal{I}^*\cap\mathcal{J}^*\subseteq\Thick(\I,\J)^*$.
\end{proof}

\begin{proposition}\label{Prop2ttDMLInclusion}
    Let $\mathcal{K}$ be an essentially small tt-category. Then $(\mathcal{I}\cap\mathcal{J})^* \supseteq \Thick(\mathcal{I}^*,\mathcal{J}^*)$ for all tt-ideals $\mathcal{I},\mathcal{J}\subseteq\mathcal{K}$.
\end{proposition}

\begin{proof}
    Since $\I\cap\J\subseteq\I,\J$, we have that $\I^*,\J^*\subseteq(\I\cap\J)^*$. Hence, $\Thick(\I^*,\J^*)\subseteq(\I\cap\J)^*$.
\end{proof}

However, not every essentially small tt-category satisfies this second tt-De Morgan law. To see an example of this, we first include the following lemmas which provide an alternative characterisation of the tt-annihilator ideals.

\begin{lemma}\label{LemInteriors}
    Let $\mathcal{I}$ and $\mathcal{J}$ be tt-ideals of an essentially small tt-category $\mathcal{K}$ with no nilpotent objects. Then $\mathcal{I}^*\subseteq\mathcal{J}^*$ if and only if $U(\mathcal{I})^\circ\subseteq U(\mathcal{J})^\circ$ in $\SpcKinv$, where $U(\I)=\{\P\in\SpcK~|~\I\subseteq\P\}$.
\end{lemma}

\begin{proof}
    Suppose $\mathcal{I}^*\subseteq\mathcal{J}^*$ and let $\mathcal{P}\in U(\mathcal{I})^\circ$. Then there is an object $x\in\mathcal{K}$ such that $\mathcal{P}\in\supp(x)\subseteq U(\mathcal{I})$. Hence, $\supp(x\otimes\mathcal{I}) = \supp(x)\cap\supp(\mathcal{I})=\varnothing$, which implies that $x\otimes\mathcal{I}\simeq 0$. Therefore, $x\in\mathcal{I}^*\subseteq\mathcal{J}^*$, i.e., $x\otimes\mathcal{J}\simeq 0$. Hence, all primes which do not contain $x$ must contain $\mathcal{J}$, i.e., $\supp(x)\subseteq U(\mathcal{J})$, and so, $\mathcal{P}\in\supp(x)\subseteq U(\mathcal{J})^\circ$.

    Conversely, suppose $U(\mathcal{I})^\circ\subseteq U(\mathcal{J})^\circ$ and let $x\in\mathcal{I}^*$. Then $\supp(x)\subseteq U(\mathcal{I})^\circ\subseteq U(\mathcal{J})^\circ$, and hence, $\supp(x\otimes\mathcal{J}) = \supp(x)\cap\supp(\mathcal{J})=\varnothing$. Therefore, $x\otimes\mathcal{J}\simeq0$ and $x\in\mathcal{J}^*$.
\end{proof}

\begin{lemma}\label{LemAnniffOA}
    Let $\mathcal{I}$ be a tt-ideal of an essentially small tt-category $\mathcal{K}$ with no nilpotent objects. Then $\mathcal{I}$ is a tt-annihilator ideal if and only if $\mathcal{I}=O_A=\{k\in\mathcal{K}~|~A\subseteq U(k)\}=\bigcap_{\mathcal{P}\in A}\mathcal{P}$ for some open subset $A\subseteq\SpcKinv$.
\end{lemma}

\begin{proof}
    Suppose $\mathcal{I}$ is a tt-annihilator ideal and let $x\in\mathcal{I}$. Then $U(\mathcal{I})^\circ\subseteq U(x)^\circ$ and so, $x\in O_{U(\mathcal{I})^\circ}$. Hence, $\mathcal{I}\subseteq O_{U(\mathcal{I})^\circ}$. Now let $x\in O_{U(\mathcal{I})^\circ}$. Then $U(\mathcal{I})^\circ\subseteq U(x)^\circ$ and by Lemma \ref{LemInteriors}, $\mathcal{I}^*\subseteq(x)^*$. Hence, $x\in(x)^{**}\subseteq\mathcal{I}^{**}=\mathcal{I}$, so $O_{U(\mathcal{I})^\circ}\subseteq\I$.

    Conversely, let $\mathcal{I}=O_A$ for some open subset $A\subseteq\text{Spc}(\mathcal{K})^\vee$, which is necessarily of the form $\supp(\mathcal{J})$ for some tt-ideal $\mathcal{J}\subseteq\mathcal{K}$. We claim that $\mathcal{I}=O_{\supp(\mathcal{J})}=\mathcal{J^*}$. To see this, first let $x\in\mathcal{J}^*$. Then $\supp(\mathcal{J})\subseteq U(x)$ so $x\in O_{\supp(\mathcal{J})}$. Now let $x\in O_{\supp(\mathcal{J})}$, i.e., $\supp(\mathcal{J})\subseteq U(x)$. Then $\supp(x\otimes\mathcal{J}) = \supp(x)\cap\supp(\mathcal{J})=\varnothing$, so $x\otimes\mathcal{J}\simeq0$ and $x\in\mathcal{J}^*$ as required.
\end{proof}

Recall that any radical tt-ideal $\I$ is given by the intersection of the prime tt-ideals of the closed set $U(\I)=\{\P\in\SpcKinv~|~\I\subseteq\P\}$. It is therefore intuitive that the tt-annihilator ideal $\I^*$ should be given by the intersection of the primes in the complement of $U(\I)$, namely, $\supp(\I)$.

\begin{example}\label{ExDZnotDML}
    Consider the essentially small tt-category $\DZ$, the perfect complexes of the derived category of $\mathbb{Z}$. By Thomason's result in Example \ref{ExampleThomason}, we have that $\SpcDZinv\cong\SpecZ^\vee$ which is as follows:

    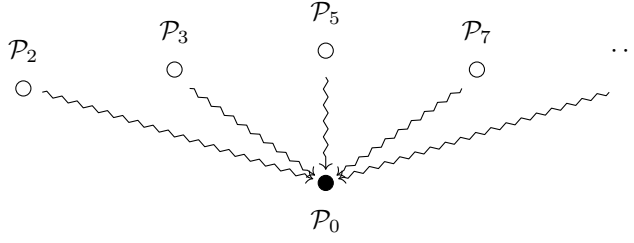
\begin{figure}[H]
    \centering
\begin{tikzpicture}[yscale=-1]
\draw[white, fill=white] (0,0) circle (.3cm);
\node[circle] at (0,0) (e) {};
\draw[fill = black] (0,0) circle (0.1cm);
\draw[draw = black] (-4,-1.25) circle (0.1cm);
\draw[draw = black] (-2,-1.5) circle (0.1cm);
\draw[draw = black] (0,-1.75) circle (0.1cm);
\draw[draw = black] (2,-1.5) circle (0.1cm);
\node at (0,.5) {$\mathcal{P}_0$};
\node at (-4,-1.75) {$\mathcal{P}_2$};
\node at (-2,-2.0) {$\mathcal{P}_3$};
\node at (0,-2.25) {$\mathcal{P}_5$};
\node at (2,-2.0) {$\mathcal{P}_7$};
\node at (4,-1.75) {$\cdots$};
\draw [->,line join=round,decorate, decoration={zigzag, segment length=6, amplitude=.9,post=lineto, post length=2pt}]  (-3.75,-1.2)  -- (e);
\draw [->,line join=round,decorate, decoration={zigzag, segment length=6, amplitude=.9,post=lineto, post length=2pt}]  (3.75,-1.2) -- (e);
\draw [->,line join=round,decorate, decoration={zigzag, segment length=6, amplitude=.9,post=lineto, post length=2pt}]  (-1.8,-1.25) -- (e);
\draw [->,line join=round,decorate, decoration={zigzag, segment length=6, amplitude=.9,post=lineto, post length=2pt}]  (1.8,-1.25)  -- (e);
\draw [->,line join=round,decorate, decoration={zigzag, segment length=6, amplitude=.9,post=lineto, post length=2pt}]  (0.0,-1.4)  -- (e);
\end{tikzpicture}\caption{The Hochster dual of the Balmer spectrum of $\DZ$, the perfect complexes of the derived category of $\mathbb{Z}$.}\label{fig:SpcDZinv}
\end{figure}

where $\P_0=\{k\in\K~|~k\otimes\mathbb{Q}\simeq0\}$ and $\P_p=\{k\in\K~|~k\otimes\mathbb{F}_p\simeq0\}$ for $p\in\mathbb{Z}$ non-zero and prime. The open subsets of $\SpecZ^\vee$ are precisely the Thomason subsets of $\SpecZ$, which are given by unions of collections of non-zero prime ideals $(p)\unlhd\mathbb{Z}$. Hence, an open set of $\SpcDZinv$ is a union of prime tt-ideals $\P_p\in\SpcDZinv$ for $p\in\mathbb{Z}$ non-zero and prime.

For the purpose of this example, we will enumerate the minimal primes as $\{\P_1, \P_2, \P_3, ...\}$ and consider the tt-ideals
\begin{equation*}
    \E=\bigcap\limits_{n=2k}\P_n ~~~\text{and}~~~ \oo=\bigcap\limits_{n=2k+1}\P_n,
\end{equation*}
which are tt-ideals (in fact tt-annihilator ideals) by Lemma \ref{LemAnniffOA}. It is not difficult to see that $\E^*=\oo$ and $\oo^*=\E$. Then 
\begin{align*}
    (\E\cap\oo)^*&=(0)^*=\DZ,\\
    \text{while } ~ \Thick(\E^*,\oo^*)&=\Thick(\oo,\E)=\mathsf{D}^{\mathrm{perf}}_\mathrm{tors}(\mathbb{Z})\neq\DZ
\end{align*}

since the torsion-free objects, for example $\mathbb{Z}$, are not objects of $\Thick(\oo,\E)$. Hence, $\DZ$ does not satisfy the second tt-De Morgan law.
\end{example}

Although we have defined the tt-De Morgan laws for all tt-ideals of a given essentially small tt-category $\K$, Lemma \ref{LemttAnnRadical} tells us that when $\K$ has no nilpotent objects, all tt-ideals which appear in the tt-De Morgan laws are radical tt-ideals. We can therefore apply Johnstone's theorem (Theorem \ref{TheoremJohnstone}) to the frame of radical tt-ideals of $\K$, giving us the following corollary.

\begin{corollary}\label{Corollary2DMLiffED}
    Let $\mathcal{K}$ be an essentially small tt-category with no nilpotent objects. Then $\mathcal{K}$ satisfies the second tt-De Morgan law if and only if $\SpcKinv$ is extremally disconnected.
\end{corollary}

\subsection{Baer Tensor-Triangulated Categories}\label{SectionBaerTT}

We now restrict our discussion to essentially small \textbf{rigid} tt-categories, i.e., essentially small tt-categories such that every object is dualizable. This extra assumption is needed to guarantee that the only idempotent objects exist in idempotent triangles as in the following definition. We then work towards giving an intrinsic definition of those essentially small rigid tt-categories which satisfy the second tt-De Morgan law. Note that by definition, a rigid tt-category does not contain any nilpotent objects, so the previous results are valid.

\begin{definition}\label{DefinitionIdempotent}
    (\cite[Definition 3.2.]{BF11}) Let $\mathcal{K}$ be an essentially small rigid tt-category. An object $e\in\mathcal{K}$ is \textbf{left idempotent} if there is an exact triangle
    \begin{equation*}
        e \rightarrow \mathbbm{1} \rightarrow f \rightarrow \Sigma e
    \end{equation*}
    such that $e\otimes f\simeq0$, where $\mathbbm{1}$ denotes the $\otimes$-unit of $\K$. In this situation, $f$ is the corresponding \textbf{right idempotent} object. Without loss of generality, we shall refer to the left idempotent objects of $\mathcal{K}$ as \textbf{idempotents}.
\end{definition}

\begin{definition}\label{DefBaertt}
    Let $\mathcal{K}$ be an essentially small rigid tt-category. Then $\mathcal{K}$ is a \textbf{Baer tt-category} if every tt-annihilator ideal of $\mathcal{K}$ is generated by an idempotent object $e\in\K$.
\end{definition}

When we have an idempotent triangle as above, $\supp(e)=U(f)$ and $\supp(f)=U(e)$. That is, $e$ and $f$ both have clopen support. This should already indicate the link between Baer tt-categories and extremally disconnected Balmer spectra, a link which will be made precise in Corollary \ref{CorollaryBaerttED}.

We claim that these Baer tt-categories are precisely the essentially small rigid tt-categories which satisfy the second tt-De Morgan law. To prove this, we require some lemmas related to the tt-annihilator ideals of a given tt-category $\K$.

\begin{lemma}\label{LemProductsEqual}
    Let $\mathcal{K}$ be an essentially small tt-category and let $\mathcal{I},\mathcal{I'},\mathcal{J},\mathcal{J'}\subseteq\mathcal{K}$ be tt-ideals. If $\mathcal{I}^*=(\mathcal{I}')^*$ and $\mathcal{J}^*=(\mathcal{J}')^*$, then
    \begin{equation*}
        (\mathcal{I\otimes J})^*=(\mathcal{I'\otimes J'})^*.
    \end{equation*}
\end{lemma}

\begin{proof}
    Let $x\in(\mathcal{I\otimes J})^*$, then for all $i\in\mathcal{I},j\in\mathcal{J}$, we have $x\otimes i\otimes j\simeq 0$. Hence, $x\otimes j\in\mathcal{I}^*$ for all $j\in\mathcal{J}$. But $\mathcal{I}^*=(\mathcal{I}')^*$, so $x\otimes i'\otimes j\simeq0$ for all $i'\in\mathcal{I'},j\in\mathcal{J}$, i.e., $x\otimes i'\in\mathcal{J}^*$ for all $i'\in\mathcal{I}'$. Then since $\mathcal{J}^*=(\mathcal{J}')^*$, $x\otimes i'\otimes j'\simeq0$ for all $i'\in\mathcal{I}',j'\in\mathcal{J}'$, i.e., $x\in(\mathcal{I'\otimes J'})^*$. The converse is shown similarly.
\end{proof}

\begin{lemma}\label{LemttIdempotents}
    Let $\mathcal{K}$ be an essentially small rigid tt-category and let $e_1,e_2\in\mathcal{K}$ be idempotent objects. Then
    \begin{equation*}
        (e_1\otimes e_2)^* = \Thick\big((e_1)^*,(e_2)^*\big).
    \end{equation*}
\end{lemma}

\begin{proof}
    Let $x\in (e_1\otimes e_2)^*$ and consider the idempotent triangle
    \begin{equation*}
        e_1\rightarrow\mathbbm{1}\rightarrow f_1\rightarrow\Sigma e_1.
    \end{equation*}
    Tensoring with $x$ gives 
    \begin{equation*}
        x\otimes e_1\rightarrow x\rightarrow x\otimes f_1\rightarrow\Sigma(x\otimes e_1).
    \end{equation*}
    Now, $x\otimes e_1\in(e_2)^*$ since $x\in (e_1\otimes e_2)^*$, and $x\otimes f_1\in(e_1)^*$. Hence, $(x\otimes e_1),(x\otimes f_1)\in\Thick\big((e_1)^*,(e_2)^*\big)$ and thus, $x\in\Thick\big((e_1)^*,(e_2)^*\big)$ by the 2 out of 3 property.

    For the other inclusion, let $x\in(e_1)^*$. Then clearly $x\in(e_1\otimes e_2)^*$, since $x\otimes e_1\simeq0$. Hence, $(e_1)^*\subseteq(e_1\otimes e_2)^*$, and similarly $(e_2)^*\subseteq(e_1\otimes e_2)^*$, so $\Thick\big((e_1)^*,(e_2)^*\big)\subseteq(e_1\otimes e_2)^*$.
\end{proof}

\begin{theorem}\label{ThmBaer2ttDML}
    Let $\mathcal{K}$ be an essentially small rigid tt-category. Then $\mathcal{K}$ is a Baer tt-category if and only if $\mathcal{K}$ satisfies the second tt-De Morgan law.
\end{theorem}

\begin{proof}
    Suppose $\mathcal{K}$ is a Baer tt-category. One inclusion of the second tt-De Morgan law always holds (Proposition \ref{Prop2ttDMLInclusion}), and for the other inclusion, first observe that $(\mathcal{I}\cap\mathcal{J})^*\subseteq(\mathcal{I}\otimes\mathcal{J})^*$ since $\mathcal{I}\otimes\mathcal{J}\subseteq\mathcal{I}\cap\mathcal{J}$ for all tt-ideals $\mathcal{I},\mathcal{J}\subseteq\mathcal{K}$. Now, since $\mathcal{K}$ is a Baer tt-category, $\mathcal{I}^*=(e_1)=(f_1)^*$ and $\mathcal{J}^*=(e_2)=(f_2)^*$, so by Lemma \ref{LemProductsEqual} we have $(\mathcal{I}\otimes\mathcal{J})^*=(f_1\otimes f_2)^*$, which in turn is equal to $\Thick\big((f_1)^*,(f_2)^*\big)=\Thick(\I^*,\J^*)$ by Lemma \ref{LemttIdempotents}.
    
    Conversely, suppose that $\mathcal{K}$ satisfies the second tt-De Morgan law. Since $\mathcal{K}$ has no nilpotent objects, $\mathcal{I}\cap\mathcal{I}^*=0$ for every tt-ideal $\mathcal{I}\subseteq K$. Hence,
    \begin{equation*}
        \Thick(\mathcal{I}^*,\mathcal{I}^{**}) = (\mathcal{I}\cap\mathcal{I}^*)^* = 0^* = \mathcal{K}.
    \end{equation*}
    Again, $\mathcal{I}^*\cap\mathcal{I}^{**}=0$, so $\supp(\I^*)\cap\supp(\I^{**})=\varnothing$ and by the above, $\supp(\I^*)\cup\supp(\I^{**})=\SpcKinv$. Hence, by \cite[Theorem 2.11.]{Bal07}, $\K\simeq\I^*\oplus\I^{**}$. Therefore, $\mathbbm{1}=x\oplus y$ for $x\in\mathcal{I}^*$ and $y\in\mathcal{I}^{**}$. Firstly, $x$ is idempotent since
    \begin{equation*}
        \mathbbm{1}=x\oplus y \Longleftrightarrow x=(x\otimes x)\oplus(x\otimes y) \Longleftrightarrow x=x\otimes x,
    \end{equation*}
    using the fact that $y\in\mathcal{I}^{**}$. Secondly, $x$ generates the tt-ideal $\mathcal{I}^*$ since for $z\in\mathcal{I}^*$, 
    \begin{equation*}
        z = z\otimes\mathbbm{1} = (z\otimes x)\oplus(z\otimes y) = z\otimes x,
    \end{equation*}
    again using that $y\in\mathcal{I}^{**}$. Every tt-annihilator ideal of $\mathcal{K}$ is thus generated by an idempotent object, i.e., $\mathcal{K}$ is a Baer tt-category.
\end{proof}

The following result then follows immediately from Corollary \ref{Corollary2DMLiffED}.

\begin{corollary}\label{CorollaryBaerttED}
    Let $\K$ be an essentially small rigid tt-category. Then the following are equivalent:
    \begin{enumerate}
        \item $\K$ is a Baer tt-category.
        \item $\K$ satisfies the second tt-De Morgan law.
        \item $\SpcKinv$ is extremally disconnected.
    \end{enumerate}
\end{corollary}

\subsection{Examples}

We will now see some examples of Baer tt-categories and essentially small tt-categories which are not Baer. These examples in particular highlight the usefulness of the previous corollary, that is, it is much easier to check for a certain topological property on the Balmer spectrum of an essentially small tt-category, rather than working directly with the tt-ideals.

\begin{example}\label{ExPLocal}
    Consider the $p$-local finite stable homotopy category, $\SHCp$. The Hochster dual of the Balmer spectrum of $\SHCp$ is as follows, by the classic result of \cite{HS98}:
    \begin{figure}[H]
    \centering
\begin{tikzpicture}[yscale=1]
\draw[draw = black] (-4,0) circle (0.1cm);
\draw[fill = gray!50] (0,0) circle (0.1cm);
\draw[fill = gray!50] (2,0) circle (0.1cm);
\draw[fill = black] (4,0) circle (0.1cm);
\node at (-4.35,-0.5) {$(0)=\P_\infty$};
\node at (-2,0) {$\cdots$};
\node at (0,-0.5) {$\P_3$};
\node at (2,-0.5) {$\P_2$};
\node at (4,-0.5) {$\P_1$};
\draw [->,line join=round,decorate, decoration={zigzag, segment length=6, amplitude=.9,post=lineto, post length=2pt}]  (-3.5,0) -- (-2.5,0);
\draw [->,line join=round,decorate, decoration={zigzag, segment length=6, amplitude=.9,post=lineto, post length=2pt}]  (-1.5,0) -- (-0.5,0);
\draw [->,line join=round,decorate, decoration={zigzag, segment length=6, amplitude=.9,post=lineto, post length=2pt}]  (0.5,0) -- (1.5,0);
\draw [->,line join=round,decorate, decoration={zigzag, segment length=6, amplitude=.9,post=lineto, post length=2pt}]  (2.5,0) -- (3.5,0);
\end{tikzpicture}\caption{The Hochster dual of the Balmer spectrum of the $p$-local finite stable homotopy category, $\SHCp$.}\label{fig:Spc(SHp)inv}
\end{figure}
where $\P_1$ is a closed point and the direction of closure is to the right. Here, $\P_n$ denotes the tt-ideal $\{X\in\SHCp~|~K(n-1)_*(X)\simeq0\}$ for $n\in\mathbb{N}\cup\{\infty\}$, where $K(n)$ represents the $n^{\mathrm{th}}$ Morava $K$-theory. This space is clearly extremally disconnected, since the closure of any open set is the entire space.

Hence, by Corollary \ref{CorollaryBaerttED}, $\SHCp$ is a Baer tt-category. To see this, recall that for two finite $p$-local spectra $X$ and $Y$, $X\wedge Y\simeq 0$ implies that $X\simeq0$ or $Y\simeq0$. Hence, for a tt-ideal $\I\subseteq\SHCp$, if $\I$ contains at least one non-zero spectrum we must have $\I^*\simeq(0)$, and otherwise we must have $\I^*\simeq\SHCp$. There are thus only two tt-annihilator ideals of $\SHCp$, namely $(0)$ and $\SHCp$ itself, which are of course idempotent generated by $0$ and $\mathbbm{1}\simeq S_{(p)}$.

Furthermore, this tells us that $\SHCp$ satisfies the second tt-De Morgan law. To see this, let $\I,\J\subseteq\SHCp$ be tt-ideals. Since every proper tt-ideal of $\SHCp$ is prime, without loss of generality we assume that $\I\subseteq\J$, so $\I\cap\J\simeq\I$ and $\Thick(\I,\J)\simeq\J$. There are two cases to consider:
\begin{enumerate}
    \item If $\I\not\simeq(0)$ then $\I^*\simeq\J^*\simeq (0)$. Hence,
    \begin{equation*}
        (\I\cap\J)^*\simeq\I^*\simeq(0)\simeq\Thick(\I^*,\J^*).
    \end{equation*}
    \item If $\I\simeq(0)$ then $\I^*\simeq\SHCp$. Hence,
    \begin{equation*}
        (\I\cap\J)^*\simeq\I^*\simeq\SHCp\simeq\Thick(\SHCp,\J^*)\simeq\Thick(\I^*,\J^*).
    \end{equation*}
\end{enumerate}
\end{example}

\begin{example}\label{ExampleDR}
    Now consider the ring 
    \begin{equation*}
        R=\frac{k[x_1,x_2,...]_{(x_1,x_2,...)}}{(x_ix_j~|~i\neq j)}
    \end{equation*}
    from \cite[Section 3]{BGS26}. The spectrum of this ring, and hence the Balmer spectrum of $\mathsf{D}^{\mathrm{perf}}(R)$, is isomorphic to $\SpecZ^\vee$ as shown in \cite{BGS26}. We therefore have $\mathrm{Spc(}\mathsf{D}^{\mathrm{perf}}(R))^\vee\cong\SpecZ$, and since $\SpecZ$ is extremally disconnected by Example \ref{ExampleZ}, $\mathsf{D}^{\mathrm{perf}}(R)$ is a Baer tt-category by Corollary \ref{CorollaryBaerttED}.
    
    By Lemma \ref{LemAnniffOA}, every tt-annihilator ideal of $\mathsf{D}^{\mathrm{perf}}(R)$ is of the form $\bigcap_{\mathcal{P}\in A}\mathcal{P}$ for some open subset $A\subseteq\mathrm{Spc}(\mathsf{D}^{\mathrm{perf}}(R))^\vee$. But the minimal prime tt-ideal of $\mathsf{D}^{\mathrm{perf}}(R)$ is the zero ideal, which is dense in $\mathrm{Spc}(\mathsf{D}^{\mathrm{perf}}(R))^\vee$, so the only tt-annihilator ideals of $\mathsf{D}^{\mathrm{perf}}(R)$ are $(0)$ and $\mathsf{D}^{\mathrm{perf}}(R)$ itself. The same reasoning as Example \ref{ExPLocal} can therefore be applied to directly show that $\mathsf{D}^{\mathrm{perf}}(R)$ is a Baer tt-category and that the second tt-De Morgan law holds.
\end{example}

\begin{example}\label{ExDZ}
 Recall Example \ref{ExDZnotDML} where it was shown that $\DZ$ does not satisfy the second tt-De Morgan Law. By the previous results, $\DZ$ is not a Baer tt-category, and $\SpcDZinv$ (Figure \ref{fig:SpcDZinv}) is not extremally disconnected. We can show both of these facts directly.
 
 Firstly, the only idempotent objects of $\DZ$ are 0 and the tensor unit $\mathbb{Z}$. Hence, the tt-annihilator ideal $\P_2$ (which is a tt-annihilator ideal by Lemma \ref{LemAnniffOA}) is clearly not idempotent generated. $\DZ$ is therefore not a Baer tt-category.
 
 Secondly, consider the open set $\{\P_2\}\subseteq\SpcDZinv$. The closure of this set is $\{\P_0,\P_2\}$, which is clearly not an open set. The space $\SpcDZinv$ is thus not extremally disconnected.  
\end{example}

The last two examples that we include here are non-rigid essentially small tt-categories. The first makes use of Corollary \ref{Corollary2DMLiffED} and shows that satisfying the second tt-De Morgan law is still equivalent to $\SpcKinv$ being extremally disconnected in the non-rigid case, while the second example shows why the rigidity assumption is necessary in our definition of Baer tt-category.

\begin{example}
    Consider the subcategory of compact objects of the derived category of global representations for the family of elementary abelian $p$-groups, $\Ep$. By \cite[Theorem 6.13.]{BBPSW25}, $\SpcEp\cong\SpecZ^\vee$, so by taking the Hochster dual we get that $\SpcEpinv\cong\SpecZ$. This space is extremally disconnected, so $\Ep$ satisfies the second tt-De Morgan Law by Corollary \ref{Corollary2DMLiffED}. To see this, note that $\Ep$ does not have any nilpotent objects, so by Lemma \ref{LemAnniffOA} every tt-annihilator ideal is of the form $\bigcap_{\mathcal{P}\in A}\mathcal{P}$ for some open subset $A\subseteq\SpcEpinv$. As in the case of Example \ref{ExampleDR}, the zero ideal is dense in $\SpcEpinv$, so again the only tt-annihilator ideals of $\Ep$ are $(0)$ and $\Ep$ itself. $\Ep$ thus trivially satisfies the second tt-De Morgan Law.
\end{example}

\begin{example}
    Let $\C$ be the category $1\rightarrow2\rightarrow3$ and let $\K$ be the (non-rigid) essentially small tt-category of perfect complexes of $\C$-modules, $\mathsf{D}^\mathrm{perf}(\C)$. Here, the tensor product is given by
    \begin{align*}
        \mathrm{mod}(\C)\otimes\mathrm{mod}(\C)&\rightarrow\mathrm{mod}(\C),\\
        (F,G)&\mapsto F\otimes G\colon\C^\mathrm{op}\times\C^\mathrm{op}\rightarrow\mathrm{mod}(k)\times\mathrm{mod}(k)\xrightarrow{\otimes}\mathrm{mod}(k)
    \end{align*}
    for some field $k$. We set $P_i:=k\C(-,i)$ and get the following projective objects:
    \begin{align*}
        &P_1 = k \leftarrow 0 \leftarrow 0, \\
        &P_2 = k \leftarrow k \leftarrow 0, \\
        &P_3 = k \leftarrow k \leftarrow k.
    \end{align*}
    Note that $P_3$ is the tensor unit $\mathbbm{1}$ of $\K$. We have a short exact sequence
    \begin{equation*}
        P_1\longrightarrow P_3\longrightarrow I_2 = \sfrac{P_3}{P_1} = (0 \leftarrow k \leftarrow k)
    \end{equation*}
    which gives rise to the following system:
    \begin{center}
\begin{tikzcd}
S_2 & = & 0 & k \arrow[l] \arrow[d, hook]      & 0 \arrow[l] \arrow[d, hook]      \\
I_2 & = & 0 & k \arrow[l] \arrow[d, two heads] & k \arrow[l] \arrow[d, two heads] \\
S_3 & = & 0 & 0 \arrow[l]                      & k \arrow[l]                     
\end{tikzcd}
    \end{center}
    where $S_2$ and $S_3$ are simple objects of $\K$. We thus have a short exact sequence
    \begin{equation*}
        S_2 \longrightarrow I_2 \longrightarrow S_3
    \end{equation*}
    where $S_2$ and $S_3$ are idempotent objects, but $I_2\neq\mathbbm{1}$, so this is not an idempotent triangle in the sense of Definition \ref{DefinitionIdempotent}. 
\end{example}

At the beginning of Section \ref{SectionBaerTT}, we made the assumption that all idempotent objects exist as left or right idempotents in idempotent triangles. This is true for rigid tt-categories by \cite[Proposition 3.1.]{BF11}, however this example demonstrates that this assumption is not valid for non-rigid tt-categories, thus showing the necessity of the rigidity condition in our definition of Baer tt-category.

\section{Baer Reflection Functors}\label{SectionReflection}

Thus far, we have seen the connections between the topological property of being extremally disconnected, the frame-theoretic property of satisfying the second De Morgan law, and the algebraic/tt-geometric property of annihilator ideals being idempotent generated. This section is dedicated to constructing commutative rings and tt-categories which have these properties. However, we begin topologically.

Given a topological space $X$, Section 4 of \cite{Rum09} gives a construction to produce an extremally disconnected space from $X$, known as the \textbf{absolute} of $X$ and denoted by $\breve{X}$, along with a continuous comparison map $a:\breve{X}\rightarrow X$. This construction therefore gives us a method for producing topological spaces whose frame of open sets satisfy the second De Morgan Law. This is done by forcing every regular open set to be clopen (see discussion in Section \ref{SectionBaerRings}).

\begin{construction}\label{ConstructionRump}
    (\cite{Rum09}) Let $X$ be a topological space and let $\Phi(X)$ denote the family of finite collections of regular open subsets of $X$ such that for each $\mathcal{A}\in\Phi(X)$, the regular open subsets of $\mathcal{A}$ are pairwise disjoint and the union of all subsets of $\mathcal{A}$ is dense in $X$. We write that $\mathcal{A}\leq\mathcal{A}'$ when $\mathcal{A}$ is finer than $\mathcal{A}'$, that is, every set $U\in\mathcal{A}$ is contained in some set $V\in\mathcal{A}'$. Now define $\overline{\mathcal{A}}:=\coprod\{\overline{U}~|~U\in \mathcal{A}\}$. Then
\begin{equation*}
    \breve{X} := \lim_{\substack{\longleftarrow\\\mathcal{A}\in\Phi(X)}}\overline{\mathcal{A}}
\end{equation*}
is the absolute of $X$.

Observe that for two partitions $\A\leq\A'$, we have a map
    \begin{equation*}
    \coprod\limits_{U\in\A}\overline{U}\longrightarrow\coprod\limits_{V\in\A'}\overline{V},
    \end{equation*}
    since for every $U_i\in\A$ we have the inclusions
    \begin{equation*}
        \overline{U_i}\xhookrightarrow{}\overline{V_j}\xhookrightarrow{}\coprod\limits_{V\in\A'}\overline{V}
    \end{equation*}
    for some regular open set $V_j\in\A'$. This gives rise to a continuous map into the space $X$ from its absolute, $a\colon\breve{X}\rightarrow X$, which will be of use when computing examples.
\end{construction}

\begin{theorem}
    (\cite[Theorem 1]{Rum09}) Let $X$ be a topological space. Then the absolute of $X$ is extremally disconnected.
\end{theorem}

\begin{example}
    Recall the 3 point topological space $X$ of Example \ref{Example3set}. The only non-trivial regular open subsets of $X$ are $\{1\}$ and $\{2\}$, and the union of these two subsets is dense in $X$. Hence, $\Phi(X)=\Big\{\big\{\{1\},\{2\}\big\},X\Big\}$, and since $\big\{\{1\},\{2\}\big\}$ is finer than $X$, we get
    \begin{equation*}
        \breve{X} = \lim_{\substack{\longleftarrow\\\mathcal{A}\in\Phi(X)}}\overline{\mathcal{A}} = \overline{\big\{\{1\},\{2\}\big\}} = \overline{\{1\}}\coprod\overline{\{2\}} = \{0,1\}\coprod\{0,2\}.
    \end{equation*}
    Hence, $\breve{X}$ is homeomorphic to the disjoint union of two copies of the Sierpiński space, which is represented in the following diagram. This space can be easily verified to be extremally disconnected.

    \begin{figure}[H]
    \centering
\begin{tikzpicture}[yscale=-1]
\draw[white, fill=white] (-1,0) circle (.3cm);
\node[circle] at (-1,0) (e) {};
\draw[white, fill=white] (1,0) circle (.3cm);
\node[circle] at (1,0) (f) {};
\draw[fill = black] (-1,0) circle (0.1cm);
\draw[fill = black] (1,0) circle (0.1cm);
\draw[draw = black] (-1,-1.5) circle (0.1cm);
\draw[draw = black] (1,-1.5) circle (0.1cm);
\node at (-1,.5) {$0_1$};
\node at (1,.5) {$0_2$};
\node at (-1,-2.0) {$1$};
\node at (1,-2.0) {$2$};
\draw [->,line join=round,decorate, decoration={zigzag, segment length=6, amplitude=.9,post=lineto, post length=2pt}]  (-1,-1.25) -- (e);
\draw [->,line join=round,decorate, decoration={zigzag, segment length=6, amplitude=.9,post=lineto, post length=2pt}]  (1,-1.25)  -- (f);
\end{tikzpicture}\caption{The absolute of the three point set with the excluded point topology, which is given by the disjoint union of two copies of the Sierpiński space, $\mathbbm{2}\coprod\mathbbm{2}$.}\label{fig:2xSierpiński}
\end{figure}
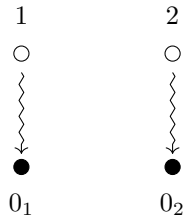
Here, the comparison map $a:\breve{X}\rightarrow X$ will map the points $0_1$ and $0_2$ in $\breve{X}$ to the point $0$ in $X$.
\end{example}

\begin{example}\label{ExAbsSpecZinv}
    In Example \ref{ExDZ} it was shown that $X=\SpecZ^\vee$ (Figure \ref{fig:SpcDZinv}) is not extremally disconnected. We can therefore compute its absolute, $\breve{X}$. As mentioned above, we have a map $a\colon\breve{X}\rightarrow X$, so computing $\breve{X}$ reduces to computing $a^{-1}(\P_0)$ and $a^{-1}(\P_p)$ for $p$ non-zero and prime.

    First looking at $a^{-1}(\P_0)$, we have the following pull-back square:
    \begin{center}
\begin{tikzcd}
a^{-1}(\P_0) \arrow[rr] \arrow[dd] &  & \displaystyle{\breve{X}=\lim_{\substack{\longleftarrow\\\mathcal{A}\in\Phi(X)}}\overline{\mathcal{A}}} \arrow[dd] \\
                                   &  &                                                                                                    \\
\P_0 \arrow[rr, hook]              &  & X.                                                                                                 
\end{tikzcd}
    \end{center}
    We therefore have that
    \begin{equation*}
        a^{-1}(\P_0) = \Big(\lim_{\substack{\longleftarrow\\\mathcal{A}\in\Phi(X)}}\overline{\mathcal{A}} \Big) \cap \{\P_0\} = \lim_{\substack{\longleftarrow\\\mathcal{A}\in\Phi(X)}}*~.
    \end{equation*}
    As mentioned in Example \ref{ExDZnotDML}, an open set of $X$ is given by any collection of minimal primes. All such collections are in fact regular open sets except for the set of all minimal primes. Hence, by relabelling the minimal primes as $\P_1,\P_2,\P_3,...$, a partition $\mathcal{A}\in\Phi(X)$ is equivalent to a partition of the natural numbers. Therefore, taking the limit of a point over finer partitions of the natural numbers, we have that $a^{-1}(\P_0)\cong\beta\mathbb{N}$, the Stone-\v Cech compactification of the natural numbers.

    We now move on to $a^{-1}(\P_p)$. Given a partition $\A'\in\Phi(X)$, the regular open sets $V\in\A'$ are pairwise disjoint, so there is a unique set $V_j$ such that $\P_p\in V_j$. Then for a finer partition $\A\leq\A'$, there is again a unique regular open set $U_i\in\A$ such that $\P_p\in U_i\subseteq V_j$. Therefore, since $\P_p$ only appears once in each partition, we must necessarily have that $a^{-1}(\P_p)=\{\P_p\}$.

    It remains to clarify the topology of $\breve{X}$. We again relabel the minimal primes as $\P_1,\P_2,\P_3,\ldots$. Then the closure of a minimal prime $\P_i$ is given by
\begin{equation*}
    \overline{\{\P_i\}}^{\breve{X}} = \bigcap_{\substack{U\text{ regular open}\\\text{s.t. }\P_i\in\overline{U}}}\overline{U}^{X} = \{\P_i,i\}
\end{equation*}
where $i\in\mathbb{N}\subseteq\beta\mathbb{N}$ is an open point. This is because $\{\P_i,\bigcup_{j\neq i}\P_j\}$ is a valid partition of regular open subsets of $X$, and the closure of $\{\P_i\}$ in $X$ is $\{\P_i,\P_0\}$. All together, $\breve{X}$ can be represented by the following diagram:

    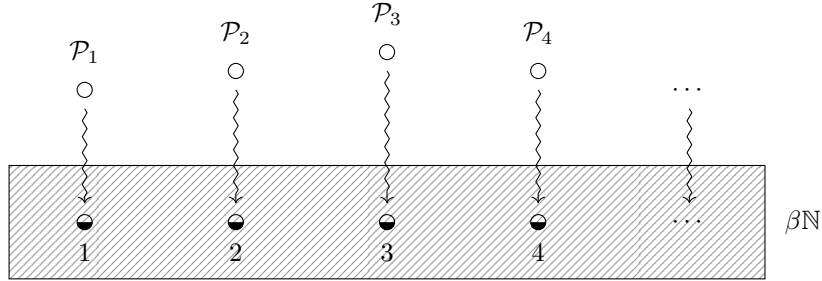
\begin{figure}[H]
    \centering
\begin{tikzpicture}[yscale=-1]
\draw[pattern=north east lines, pattern color=gray!75] (-5,-.75) rectangle (5,.75);
\draw[draw = black] (-4,-1.75) circle (0.1cm);
\draw[draw = black] (-2,-2) circle (0.1cm);
\draw[draw = black] (0,-2.25) circle (0.1cm);
\draw[draw = black] (2,-2) circle (0.1cm);
\fill[black] (-4,0) -- ++(0:0.1cm) arc (0:180:0.1cm) -- cycle;
\draw (-4,0) circle (0.1cm);
\fill[black] (-2,0) -- ++(0:0.1cm) arc (0:180:0.1cm) -- cycle;
\draw (-2,0) circle (0.1cm);
\fill[black] (0,0) -- ++(0:0.1cm) arc (0:180:0.1cm) -- cycle;
\draw (0,0) circle (0.1cm);
\fill[black] (2,0) -- ++(0:0.1cm) arc (0:180:0.1cm) -- cycle;
\draw (2,0) circle (0.1cm);
\node at (-4,-2.25) {$\mathcal{P}_1$};
\node at (-2,-2.5) {$\mathcal{P}_2$};
\node at (0,-2.75) {$\mathcal{P}_3$};
\node at (2,-2.5) {$\mathcal{P}_4$};
\node at (4,-1.75) {$\cdots$};
\node at (-4,.4) {$1$};
\node at (-2,.4) {$2$};
\node at (0,.4) {$3$};
\node at (2,.4) {$4$};
\node at (4,0) {$\cdots$};
\node at (5.5,0) {$\beta\mathbb{N}$};
\draw [->,line join=round,decorate, decoration={zigzag, segment length=6, amplitude=.9,post=lineto, post length=2pt}]  (-4,-1.5)  -- (-4,-.25);
\draw [->,line join=round,decorate, decoration={zigzag, segment length=6, amplitude=.9,post=lineto, post length=2pt}]  (4,-1.5) -- (4,-0.25);
\draw [->,line join=round,decorate, decoration={zigzag, segment length=6, amplitude=.9,post=lineto, post length=2pt}]  (-2,-1.75) -- (-2,-0.25);
\draw [->,line join=round,decorate, decoration={zigzag, segment length=6, amplitude=.9,post=lineto, post length=2pt}]  (2,-1.75)  -- (2,-.25);
\draw [->,line join=round,decorate, decoration={zigzag, segment length=6, amplitude=.9,post=lineto, post length=2pt}]  (0.0,-2)  -- (0,-.25);
\end{tikzpicture}\caption{The absolute of the Hochster dual of the Zariski spectrum of $\mathbb{Z}$, where $\beta\mathbb{N}$ represents the Stone-\v Cech compactification of the natural numbers. Here, the half filled circles represent the copy of $\mathbb{N}$ which is contained within $\beta\mathbb{N}$, which in this case are clopen points.}\label{fig:AbsSpecZinv}
\end{figure}

Applying the map $a$ to $\breve{X}$ preserves the minimal primes and collapses $\beta\mathbb{N}$ into the closed point $\P_0$, thus recovering the original space $X$ (Figure \ref{fig:SpcDZinv}). The open points of $\beta\mathbb{N}$ are clopen in $\breve{X}$, so to confirm that $\breve{X}$ is extremally disconnected we simply need to check that the closure of a minimal prime $\{\P_i\}$ is open. From above, $\overline{\{\P_i\}} = \{\P_i,i\}$, which is a union of two open points and is therefore open. All together, the closure of any open set is open, thus showing that $\breve{X}$ is extremally disconnected.
\end{example}

We now return to commutative Baer rings, which are studied in the paper \cite{Spe73}. Given a commutative reduced ring $R$, the author gives a construction to produce a universal commutative Baer ring containing $R$ (see Theorem \ref{TheoremBaerification} for the precise statement). We give Speed's ``Baerification'' construction below, where $A(R)$ denotes the complete Boolean algebra of annihilator ideals of a given commutative reduced ring $R$.

\begin{construction}
    (\cite{Spe73}) Let $R$ be a commutative reduced ring and define a partition of $A(R)$ to be a family $\mathcal{D}$ of elements of $A(R)$ such that for $D,E\in\mathcal{D}$ distinct we have $D\cap E=(0)$, and $(\Sigma_{D\in\mathcal{D}}D)^{**}=R$. For two partitions $\mathcal{C}$ and $\mathcal{D}$ we write $\mathcal{C}\leq\mathcal{D}$ when $\mathcal{D}$ is finer than $\mathcal{C}$, and denote by $\pi(R)$ the directed set of all finite partitions of $A(R)$.

We define a family of commutative rings $\{R_{\mathcal{D}}~|~\mathcal{D}\in\pi(R)\}$ and ring morphisms $f_{\mathcal{C},\mathcal{D}}\colon R_{\mathcal{C}}\rightarrow R_{\mathcal{D}}$ when $\mathcal{C}\leq\mathcal{D}$ as follows:
\begin{enumerate}
    \item $R_\mathcal{D}=\prod\limits_{D\in\mathcal{D}}\sfrac{R}{D^*}$ for each $\mathcal{D}\in\pi(R)$.
    \item For $\mathcal{C}\leq\mathcal{D}$ in $\pi(R)$, since $\mathcal{D}$ is a finer partition than $\mathcal{C}$ we can write $C=(\Sigma_\delta D_\delta)^{**}$ for all $C\in\mathcal{C}$. Then $C^*=\cap_\delta D_\delta^*$ by the first algebraic De Morgan law, which gives an isomorphism between $R/C^*$ and $\prod_\delta R/D_\delta^*$. Ranging over every $C\in\mathcal{C}$, we obtain a morphism
    \begin{equation*}
    f_{\mathcal{C},\mathcal{D}}\colon\prod\limits_{C\in\mathcal{C}}\sfrac{R}{C^*}\longrightarrow\prod\limits_{D\in\mathcal{D}}\sfrac{R}{D^*}.
    \end{equation*}
\end{enumerate}
The \textbf{Baerification} of the ring $R$ is then given by
\begin{equation*}
    B(R) := \colim_{\substack{\longrightarrow \\ \mathcal
    {D}\in\pi(R)}}\prod\limits_{D\in\mathcal{D}}\sfrac{R}{D^*},
\end{equation*}
and we let $\beta\colon R\rightarrow B(R)$ be the embedding of $R$ as a subring into its Baerification $B(R)$.
\end{construction}

\begin{definition}
    (\cite[Definition 1.2.]{Spe73}) Let $R,S$ be commutative rings. A ring morphism $\phi\colon R\rightarrow S$ is $\mathbf{\rho}$\textbf{-compatible} if for all subsets $A,B\subseteq R$ such that $A^*=B^*$, we have that $\phi(A)^*=\phi(B)^*$. In the case that $R$ and $S$ are Baer rings, we call $\phi$ a \textbf{Baer morphism}.
\end{definition}

\begin{theorem}\label{TheoremBaerification}
   (\cite[Theorem 2.5.]{Spe73}) Let $R$ be a commutative reduced ring. Then there is a commutative Baer ring $B(R)$ and a $\rho$-compatible ring monomorphism $\beta\colon R\rightarrow B(R)$ which are universal in the category of commutative Baer rings, that is, the functor $B\colon\mathsf{CRed}\rightarrow \mathsf{CBaer}$ is left adjoint to the forgetful functor $U\colon\mathsf{CBaer}\rightarrow \mathsf{CRed}$.
\end{theorem}

\begin{example}
    Again consider the commutative reduced ring $R=\mathbb{Z}[x]/(x^2-1)$, which was previously shown to not be a Baer ring. The only non-trivial annihilator ideals of $R$ are $(x+1)=(x-1)^*$ and $(x-1)=(x+1)^*$. Now, following the Baerification construction, we can only produce a one element partition, $\mathcal{C}=\{R\}$, and a two element partition, $\mathcal{D}=\{(x+1),(x-1)\}$. We thus define the rings
    \begin{align*}
        R_{\mathcal{C}} &:= \sfrac{R}{R^*} = \sfrac{R}{(0)} = R,\\
        \text{and } ~~~ R_{\mathcal{D}} &:= \sfrac{R}{(x+1)^*} \times \sfrac{R}{(x-1)^*} = \sfrac{R}{(x-1)} \times \sfrac{R}{(x+1)} \cong \sfrac{\mathbb{Z}[x]}{(x+1)} \times \sfrac{\mathbb{Z}[x]}{(x-1)},
    \end{align*}
    and the inclusion
    \begin{equation*}
        f_{\mathcal{C},\mathcal{D}}\colon\sfrac{\mathbb{Z}[x]}{(x^2-1)}\longrightarrow \sfrac{\mathbb{Z}[x]}{(x+1)} \times \sfrac{\mathbb{Z}[x]}{(x-1)}.
    \end{equation*}
    Taking the colimit of this family of commutative rings, we get $\mathbb{Z}[x]/(x+1) \times \mathbb{Z}[x]/(x-1)$, which is in fact isomorphic to $\mathbb{Z}\times\mathbb{Z}$. Hence, the Baerification of the ring $R=\mathbb{Z}[x]/(x^2-1)$ is $\mathbb{Z}\times\mathbb{Z}$, which we can confirm to be a Baer ring using Example \ref{ExampleZ} and the fact that the product of Baer rings is itself a Baer ring. One can confirm that $\mathrm{Spec}(\mathbb{Z}\times\mathbb{Z})\cong\SpecZ\coprod\SpecZ$ (below) is the absolute of $\mathrm{Spec}\big(\mathbb{Z}[x]/(x^2-1)\big)$ (Figure \ref{fig:2xSpecZglued}) in the sense of Rump's construction.
    
\begin{figure}[H]
    \centering
\begin{tikzpicture}[yscale=1]
\draw[white, fill=white] (4,0) circle (.3cm);
\node[circle] at (4,0) (e) {};
\draw[draw = black] (4,0) circle (0.1cm);
\draw[fill = black] (1,-1.25) circle (0.1cm);
\draw[fill = black] (3,-1.5) circle (0.1cm);
\draw[fill = black] (5,-1.5) circle (0.1cm);
\node at (4,.5) {$(0)$};
\node at (1,-1.75) {$(2)$};
\node at (3,-2.0) {$(3)$};
\node at (5,-2.0) {$(5)$};
\node at (7,-1.75) {$\cdots$};
\draw [->,line join=round,decorate, decoration={zigzag, segment length=6, amplitude=.9,post=lineto, post length=2pt}]  (e) -- (1.25,-1.2);
\draw [->,line join=round,decorate, decoration={zigzag, segment length=6, amplitude=.9,post=lineto, post length=2pt}]  (e) -- (3.1,-1.25);
\draw [->,line join=round,decorate, decoration={zigzag, segment length=6, amplitude=.9,post=lineto, post length=2pt}]  (e)  -- (6.8,-1.2);
\draw [->,line join=round,decorate, decoration={zigzag, segment length=6, amplitude=.9,post=lineto, post length=2pt}]  (e)  -- (4.9,-1.25);
\draw[white, fill=white] (-4,0) circle (.3cm);
\node[circle] at (-4,0) (e) {};
\draw[draw = black] (-4,0) circle (0.1cm);
\draw[fill = black] (-1,-1.25) circle (0.1cm);
\draw[fill = black] (-3,-1.5) circle (0.1cm);
\draw[fill = black] (-5,-1.5) circle (0.1cm);
\node at (-4,.5) {$(0)$};
\node at (-1,-1.75) {$(2)$};
\node at (-3,-2.0) {$(3)$};
\node at (-5,-2.0) {$(5)$};
\node at (-7,-1.75) {$\cdots$};
\draw [->,line join=round,decorate, decoration={zigzag, segment length=6, amplitude=.9,post=lineto, post length=2pt}]  (e) -- (-1.25,-1.2);
\draw [->,line join=round,decorate, decoration={zigzag, segment length=6, amplitude=.9,post=lineto, post length=2pt}]  (e) -- (-3.1,-1.25);
\draw [->,line join=round,decorate, decoration={zigzag, segment length=6, amplitude=.9,post=lineto, post length=2pt}]  (e)  -- (-6.8,-1.2);
\draw [->,line join=round,decorate, decoration={zigzag, segment length=6, amplitude=.9,post=lineto, post length=2pt}]  (e)  -- (-4.9,-1.25);
\end{tikzpicture}\caption{The Zariski spectrum of the Baerification of the ring $\mathbb{Z}[x]/(x^2-1)$, which is given by the disjoint union of two copies of the Zariski spectrum of $\mathbb{Z}$, $\SpecZ\coprod\SpecZ$.}\label{fig:2xSpecZ}
\end{figure}
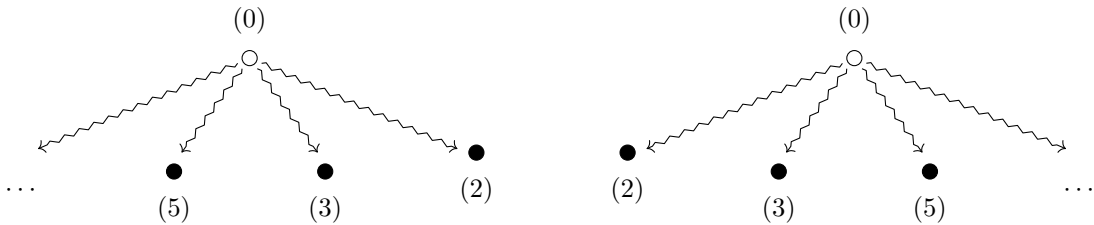
\end{example}

Rump and Speed's constructions depend purely on the underlying complete Boolean algebras of regular open subsets and annihilator ideals respectively. We now show that the collection of tt-annihilator ideals of an essentially small tt-category with no nilpotent objects forms a complete Boolean algebra, which paves the way for producing an analogous Baerification construction in tt-geometry (Construction \ref{ConstructionttBaer}).

\begin{proposition}\label{PropBoolean}
    Let $\K$ be an essentially small tt-category with no nilpotent objects. Then the collection $A(\mathcal{K})$ of tt-annihilator ideals of $\mathcal{K}$ is a complete Boolean algebra.
\end{proposition}

\begin{proof}
    For tt-annihilator ideals $\mathcal{I}^*,\mathcal{J}^*\in A(\mathcal{K})$, we define the meet and join as $\mathcal{I}^*\cap\mathcal{J}^*$ and $\Thick(\mathcal{I}^*,\mathcal{J}^*)^{**}$ respectively. Note that the intersection of tt-annihilator ideals is always a tt-annihilator ideal since the first tt-De Morgan law always holds by Proposition \ref{Prop1ttDML}, and in the case that $\mathcal{K}$ satisfies the second tt-De Morgan law, the join of two tt-annihilator ideals $\I^*$ and $\J^*$ is simply given by $(\mathcal{I}\cap\mathcal{J})^*$.
    
    It remains to show that $A(\mathcal{K})$ has complements. To see this, observe that 
    \begin{align*}
        &\mathcal{I}^*\wedge\mathcal{I}^{**}=\mathcal{I}^*\cap\mathcal{I}^{**}=(0),\\
        \text{and } ~~~ &\mathcal{I}^*\vee\mathcal{I}^{**}=\Thick(\mathcal{I}^*,\mathcal{I}^{**})^{**}=(\mathcal{I}^{**}\cap\mathcal{I}^*)^*=(0)^*=\mathcal{K}.
    \end{align*}
\end{proof}

The following lemmas are required to prove that analogously to the commutative algebra case, this complete Boolean algebra of tt-annihilator ideals is isomorphic to the complete Boolean algebra of regular open subsets of $\SpcKinv$. As in the commutative algebra case presented in Section 1 of \cite{DT21}, we define the complete Boolean algebra $\Omega^{\mathrm{reg}}(\SpcKinv)$ by the relation $A\leq B$ if and only if $A\supseteq B$, $A\vee B=A\cap B$, and $A\wedge B=\Big(\overline{A\cup B}\Big)^\circ$ for $A,B\in\Omega^{\mathrm{reg}}(\SpcKinv)$.

\begin{lemma}\label{LemIntRegOpen}
    Let $X$ be a topological space and let $C\subseteq X$ be a closed set. Then $C^\circ$ is regular open.
\end{lemma}

\begin{proof}
    $C^\circ\subseteq\Big(\overline{C^\circ}\Big)^\circ$ since this is true for any open set. Conversely, $\overline{C^\circ}\subseteq C$ since $C$ is closed. Hence, $\Big(\overline{C^\circ}\Big)^\circ\subseteq C^\circ$.
\end{proof}

Recall from Lemma \ref{LemAnniffOA} that a tt-annihilator ideal $\I$ is necessarily of the form $\mathcal{I}=O_A=\{k\in\mathcal{K}~|~A\subseteq U(k)\}=\bigcap_{\mathcal{P}\in A}\mathcal{P}$ for some open subset $A\subseteq\SpcKinv$.

\begin{lemma}
    Let $A$ and $B$ be regular open subsets of $\SpcKinv$. Then $O_A\wedge O_B=O_{A\cup B}$ and $O_A\vee O_B=O_{A\cap B}$. 
\end{lemma}

\begin{proof}
    Firstly, $O_A\wedge O_B=O_A\cap O_B$ from Proposition \ref{PropBoolean}, and
    \begin{equation*}
        O_A\cap O_B = \Bigg(\bigcap_{\mathcal{P}\in A}\mathcal{P}\Bigg) \bigcap \Bigg(\bigcap_{\mathcal{P}\in B}\mathcal{P}\Bigg) = \bigcap_{\mathcal{P}\in A\cup B}\mathcal{P} = O_{A\cup B}.
    \end{equation*}
    Secondly, $A\cap B\subseteq A$ and $A\cap B\subseteq B$, so $O_A\subseteq O_{A\cap B}$ and $O_B\subseteq O_{A\cap B}$. Now suppose $C$ is an open subset of $\SpcKinv$ such that $O_A\subseteq O_C$ and $O_B\subseteq O_C$. Hence, $U(O_C)\subseteq U(O_A)$ and $U(O_C)\subseteq U(O_B)$, and by observing that $U(O_A)$ is the smallest closed subset of $\SpcKinv$ which contains $A$, i.e., $U(O_A)=\overline{A}$, we have $\overline{C}\subseteq\overline{A}$ and $\overline{C}\subseteq\overline{B}$. Therefore,
    \begin{align*}
        &C\subseteq\Big(\overline{C}\Big)^\circ\subseteq\Big(\overline{A}\Big)^\circ = A,\\
        \text{and } ~~~ &C\subseteq\Big(\overline{C}\Big)^\circ\subseteq\Big(\overline{B}\Big)^\circ = B
    \end{align*}
    since $A$ and $B$ are regular open subsets of $\SpcKinv$. Hence, $C\subseteq A\cap B$ and thus, $O_{A\cap B}\subseteq O_C$, which completes that proof that the join of $O_A$ and $O_B$ is $O_{A\cap B}$.
\end{proof}

\begin{proposition}\label{PropositionBooleanIso}
    Let $\mathcal{K}$ be an essentially small tt-category with no nilpotent objects. Then the Boolean algebra of tt-annihilator ideals is isomorphic to the Boolean algebra of regular open subsets of $\SpcKinv$.
\end{proposition}

\begin{proof}
    Define $\phi\colon\Omega^{\mathrm{reg}}(\mathrm{Spc}(\mathcal{K})^\vee)\rightarrow A(\mathcal{K})$ by $A\mapsto O_A=\{k\in\mathcal{K}~|~A\subseteq U(k)\}=\bigcap_{\mathcal{P}\in A}\mathcal{P}$, and let $A,B\in\Omega^{\mathrm{reg}}(\mathrm{Spc}(\mathcal{K})^\vee)$ such that $\phi(A)=\phi(B)$, i.e., $O_A=O_B$. Since every tt-annihilator ideal is radical by Lemma \ref{LemttAnnRadical}, we have that $U(O_A)=U(O_B)$, and from before we know that $U(O_A)=\overline{A}$ and $U(O_B)=\overline{B}$. Hence,
    \begin{equation*}
        \overline{A}=\overline{B}\implies \Big(\overline{A}\Big)^\circ=\Big(\overline{B}\Big)^\circ \implies A=B
    \end{equation*}
    since $A$ and $B$ are regular open subsets of $\SpcKinv$, so $\phi$ is therefore injective. Now, for every tt-annihilator ideal $\mathcal{I}\in A(\mathcal{K})$, $I=O_{U(\mathcal{I})^\circ}$ by Lemma \ref{LemAnniffOA}, and $U(\mathcal{I})^\circ$ is regular open by Lemma \ref{LemIntRegOpen}. Hence, $\phi$ is surjective. It remains to show that $\phi$ preserves meets and joins:
    \begin{align*}
        &\phi(A\vee B) = \phi(A\cap B) = O_{A\cap B} = O_A\vee O_B = \phi(A)\vee\phi(B),\\
        \text{and } ~~~ &\phi(A\wedge B) = \phi\bigg(\Big(\overline{A\cup B}\Big)^\circ\bigg) = O_{\big(\overline{A\cup B}\big)^\circ} = O_{A\cup B} = O_A\wedge O_B = \phi(A)\wedge \phi(B).
    \end{align*}
\end{proof}

The above isomorphism allows us to produce a Baerification construction for essentially small rigid tt-categories as follows. Recall that for any tt-category $\K$ and any open set $U$ of $\SpcKinv$, we have a localization sequence $\Gamma_U\K\rightarrow\K\rightarrow L_U\K$, where $\Gamma_U\K$ can be thought of as the piece of $\K$ supported at $U$, and $L_U\K$ the piece of $\K$ supported away from $U$.

\begin{construction}\label{ConstructionttBaer}
    Let $\mathcal{K}$ be an essentially small rigid tt-category and let $A(\mathcal{K})$ be the complete Boolean algebra of all tt-annihilator ideals of $\mathcal{K}$. A family of tt-annihilator ideals $D\subseteq A(\mathcal{K})$ is said to be a partition of $\mathcal{K}$ if for $\mathcal{D},\mathcal{E}\in D$ distinct we have $\mathcal{D}\cap\mathcal{E}=(0)$, and $\Thick(\mathcal{D}~|~\mathcal{D}\in D)^{**}=\K$. For two partitions $C,D\in\pi(\mathcal{K})$, we write $C\leq D$ when $D$ is finer than $C$, and denote by $\pi(\mathcal{K})$ the directed set of all finite partitions of $A(\mathcal{K})$.

We define a family of tt-categories $\{\mathcal{K}_{D}~|~D\in\pi(\mathcal{K})\}$ and functors $F_{C,D}\colon\mathcal{K}_{C}\rightarrow \mathcal{K}_{D}$ when $C\leq D$ as follows:
\begin{enumerate}
    \item $\mathcal{K}_D=\prod\limits_{\mathcal{D}\in D}L_{\supp(\mathcal{D}^*)}\mathcal{K}$ for each $D\in\pi(\mathcal{K})$.
    \item For $C\leq D$ in $\pi(\mathcal{K})$, since $D$ is a finer partition than $C$, we have that each $\D\in D$ is contained in some $\C\in C$, giving us a functor
    \begin{equation*}
        F_{C,D}\colon\prod\limits_{\mathcal{C}\in C}L_{\supp(\mathcal{C}^*)}\mathcal{K}\longrightarrow\prod\limits_{\mathcal{D}\in D}L_{\supp(\mathcal{D}^*)}\mathcal{K}
    \end{equation*}
    induced by $L_{\supp(\mathcal{C}_i^*)}\mathcal{K}\rightarrow\prod\limits_{\D\subseteq\C_i}L_{\supp(\mathcal{D}^*)}\mathcal{K}$ for all $\C_i\in C$.
\end{enumerate}
We call the filtered colimit of the above family the \textbf{Baerification} of $\mathcal{K}$, denoted by $B(\mathcal{K})$, and let $\beta\colon\mathcal{K}\rightarrow B(\mathcal{K})$ be the inclusion functor of the subcategory $\mathcal{K}$ into its Baerification $B(\mathcal{K})$. Note that this functor will not necessarily be fully faithful, since a localization functor $\mathcal{K}\rightarrow L_U\mathcal{K}$ is not necessarily fully faithful.
\end{construction}

The following results show that firstly, the Baerification of an essentially small rigid tt-category is in fact a Baer tt-category, and secondly that it satisfies the expected universal property.

\begin{lemma}\label{LemStandard}
    Let $x\in B(\mathcal{K})$. Then there is a family of idempotents $\{e_i~|~1\leq i\leq n\}$ such that $\bigoplus_{1\leq i\leq n}e_i=\mathbbm{1}\in B(\mathcal{K})$ and a family of objects $\{k_i~|~1\leq i\leq n\}\subseteq\mathcal{K}$ such that
    \begin{equation*}
        x = \bigoplus\limits_{1\leq i\leq n}(\beta(k_i)\otimes e_i).
    \end{equation*}
    Furthermore, the idempotents $\{e_i~|~1\leq i\leq n\}$ can be represented by $(0,\mathbbm{1})\in L_{\supp(\mathcal{D}_i)}\mathcal{K}\times L_{\supp(\mathcal{D}_i^*)}\mathcal{K}$, where $\{\mathcal{D}_i~|~1\leq i\leq n\}$ is a partition of $\mathcal{K}$.
\end{lemma}

\begin{proof}
    By construction,
    \begin{equation*}
        x=(x_i)_{1\leq i\leq n}\in\prod_{1\leq i\leq n}L_{\supp(\mathcal{D}_i^*)}\mathcal{K}
    \end{equation*}
    for some partition $\{\mathcal{D}_i~|~1\leq i\leq n\}$ of $\mathcal{K}$. For each $1\leq i\leq n$, we set
    \begin{equation*}
        e_i = (\delta_{ij})_{1\leq j\leq n}\in\prod_{1\leq j\leq n}L_{\supp(\mathcal{D}_j^*)}\mathcal{K}
    \end{equation*}
    where $\delta$ is the Kronecker delta, and set $k_i = x_i$ which completes the proof.
\end{proof}

The representations of the idempotents $e_i$ from the previous lemma will hereafter be denoted by $(0,\mathbbm{1)}_{\mathcal{D}_i,\mathcal{D}_i^*}$ for convenience.

\begin{lemma}\label{LemIdempotentGenerated}
    Let $a\in\mathcal{K}$. Then the annihilator of $\beta(a)$ in $B(\mathcal{K})$ is generated by the idempotent $e_a=(0,\mathbbm{1)}_{(a)^*,(a)^{**}}$.
\end{lemma}

\begin{proof}
    We show that for $x\in\mathcal{K}$,
    \begin{equation*}
        \beta(a)\otimes x\simeq0\Longleftrightarrow e_a \otimes x\simeq x.
    \end{equation*}
    From Lemma \ref{LemStandard}, we can write $x$ as $\bigoplus_{1\leq i\leq n}(\beta(k_i)\otimes e_i)$. Hence,
    \begin{equation*}
        e_a \otimes x\simeq x \Longleftrightarrow e_a \otimes\beta(k_i)\otimes e_i\simeq\beta(k_i)\otimes e_i ~ \text{ for all } ~ 1\leq i\leq n,
    \end{equation*}
    where $e_i = (0,\mathbbm{1)}_{\mathcal{D}_i,\mathcal{D}_i^*}$ for some partition $\{\mathcal{D}_i~|~1\leq i\leq n\}$ of $\mathcal{K}$. Therefore, $\beta(k_i)\otimes e_i = (0,k_i)_{\mathcal{D}_i,\mathcal{D}_i^*}$. We now refine the partitions $\{(a)^*,(a)^{**}\}$ and $\{\mathcal{D}_i,\mathcal{D}_i^*\}$ to $\{((a)^*\cap\mathcal{D}_i)^*,~((a)^*\cap\mathcal{D}_i^*)^*,~((a)^{**}\cap\mathcal{D}_i)^*,~((a)^{**}\cap\mathcal{D}_i^*)^*\}$. We therefore have
    \begin{align*}
        e_a &= (\mathbbm{1},\mathbbm{1},0,0)_{\{((a)^*\cap\mathcal{D}_i)^*,~((a)^*\cap\mathcal{D}_i^*)^*,~((a)^{**}\cap\mathcal{D}_i)^*,~((a)^{**}\cap\mathcal{D}_i^*)^*\}},\\
        \text{and }~~~\beta(k_i)\otimes e_i &= (k_i,0,k_i,0)_{\{((a)^*\cap\mathcal{D}_i)^*,~((a)^*\cap\mathcal{D}_i^*)^*,~((a)^{**}\cap\mathcal{D}_i)^*,~((a)^{**}\cap\mathcal{D}_i^*)^*\}}.
    \end{align*}
    Hence, $e_a \otimes\beta(k_i)\otimes e_i\simeq\beta(k_i)\otimes e_i$ if and only if $k_i\simeq0$ in $L_{\supp(((a)^{**}\cap\mathcal{D}_i)^*)}\mathcal{K}$, and $\beta(a)\otimes\beta(k_i)\otimes e_i\simeq0$ if and only if $a\otimes k_i\simeq 0$ in $L_{\supp(\mathcal{D}_i^*)}\mathcal{K}$. But these conditions are equivalent to $k_i\in((a)^{**}\cap\mathcal{D}_i)^*$ and $a\otimes k_i\in \mathcal{D}_i^*$ respectively. To complete the proof we will show that these two conditions are in fact equivalent.

    First suppose $k_i\in((a)^{**}\cap\mathcal{D}_i)^*$ and let $b\in \mathcal{D}_i$. Then $a\otimes b\in (a)^{**}\cap \mathcal{D}_i$. Hence, $a\otimes b\otimes k_i\simeq 0$ and so, $a\otimes k_i\in\mathcal{D}_i^*$. Conversely, suppose $a\otimes k_i\in\mathcal{D}_i^*$ and let $b\in(a)^{**}\cap\mathcal{D}_i$. Then $b\otimes k_i\in(a)^{**}$ and $b\otimes a\otimes k_i\simeq 0$, i.e., $b\otimes k_i\in(a)^*$. Therefore, $(b\otimes k_i)^{\otimes2}\simeq 0$ and thus, $b\otimes k_i\simeq0$ since $\mathcal{K}$ has no nilpotent objects. Hence,
    \begin{equation*}
        b\in(k_i)^* \implies (a)^{**}\cap\mathcal{D}_i\subseteq(k_i)^* \implies k_i\in(k_i)^{**}\subseteq ((a)^{**}\cap\mathcal{D}_i)^*.
    \end{equation*}
\end{proof}

\begin{proposition}\label{PropBaer}
    $B(\mathcal{K})$ is a Baer tt-category.
\end{proposition}

\begin{proof}
    Let $x = \bigoplus_{1\leq i\leq n}(\beta(k_i)\otimes e_i)\in B(\mathcal{K})$. Then
    \begin{equation*}
        (x)^* = \bigcap\limits_{1\leq i\leq n}(\beta(k_i)\otimes e_i)^* = \bigcap\limits_{1\leq i\leq n}(\beta(k_i)^{**}\otimes e_i)^*,
    \end{equation*}
    with the last equality following from Lemma \ref{LemProductsEqual}. Hence, using Lemma \ref{LemIdempotentGenerated}, one can see that $(x)^*$ is idempotent generated.

    Now let $\mathcal{I}$ be a tt-ideal of $B(\mathcal{K})$. Then
    \begin{equation*}
        \mathcal{I}^* = \bigcap\limits_{k\in\mathcal{I}}(k)^* = \bigcap\limits_{k\in\mathcal{I}}(e_k)
    \end{equation*}
    where $e_k$ is some idempotent object by the above argument. We know that each $e_k$ is of the form $(0,\mathbbm{1)}_{\mathcal{D}_k,\mathcal{D}_k^*}$ for some tt-annihilator ideals $\mathcal{D}_k,\D_k^*\subseteq\K$. Hence,
    \begin{equation*}
        \bigcap\limits_{k\in\mathcal{I}}(e_k) = \Big((0,\mathbbm{1)}_{\bigcap\limits_{k\in\mathcal{I}}\mathcal{D}_k,~\bigcap\limits_{k\in\mathcal{I}}\mathcal{D}_k^*}\Big),
    \end{equation*}
    using the fact that the intersection of tt-annihilator ideals is again a tt-annihilator ideal by Proposition \ref{PropBoolean}. $\mathcal{I}^*$ is thus generated by an idempotent object and hence, $B(\mathcal{K})$ is a Baer tt-category.
\end{proof}

\begin{definition}
    Let $\mathcal{K,L}$ be essentially small tt-categories. A functor $F\colon\mathcal{K}\rightarrow\mathcal{L}$ is \textbf{annihilator compatible} if for all tt-ideals $\mathcal{I},\mathcal{J}\subseteq\mathcal{K}$ such that $\mathcal{I}^*=\mathcal{J}^*$, we have that $F(\mathcal{I})^*=F(\mathcal{J})^*$. In the case that $\mathcal{K,L}$ are Baer tt-categories, we call $F$ a \textbf{Baer functor}.
\end{definition}

\begin{theorem}\label{ThmttUniversal}
    Let $\mathcal{K}$ be an essentially small rigid tt-category. Then there is a Baer tt-category $B(\mathcal{K})$ and an annihilator compatible functor $\beta\colon\mathcal{K}\rightarrow B(\mathcal{K})$ such that for all annihilator compatible functors $F\colon\mathcal{K}\rightarrow\mathcal{B}$ from $\mathcal{K}$ to a Baer tt-category $\mathcal{B}$, there is a unique functor $\overline{F}\colon B(\mathcal{K})\rightarrow\mathcal{B}$ such that $\overline{F}\circ\beta=F$.
\end{theorem}

\begin{proof}
    Our construction defines the Baer tt-category $B(\mathcal{K})$ and the functor $\beta\colon\mathcal{K}\rightarrow B(\mathcal{K})$, and Proposition \ref{PropBaer} shows that $B(\mathcal{K})$ is a Baer tt-category. To see that $\beta$ is annihilator compatible, let $\mathcal{I},\mathcal{J}$ be tt-ideals of $\mathcal{K}$ such that $\mathcal{I}^*=\mathcal{J}^*$. Then $\beta(\mathcal{I})^*=(e)$ for some idempotent object $e\in B(\mathcal{K})$, and $e=(0,\mathbbm{1})_{\mathcal{I}^*,\mathcal{I}^{**}} = (0,\mathbbm{1})_{\mathcal{J}^*,\mathcal{J}^{**}}$. Hence, $\beta(\mathcal{I})^* = \beta(\mathcal{J})^*$.

    Now let $F\colon\mathcal{K}\rightarrow\mathcal{B}$ be an annihilator compatible functor from $\mathcal{K}$ to a Baer tt-category $\mathcal{B}$. We define the functor $\overline{F}\colon B(\mathcal{K})\rightarrow\mathcal{B}$ as follows: for $k\in\mathcal{K}$ we set $\overline{F}(\beta(k)) = F(k)$, and for an idempotent object $e=(0,\mathbbm{1})_{\mathcal{I}^*,\mathcal{I}^{**}}\in B(\K)$ we set $\overline{F}(e) = e_{F(\mathcal{I})^*}$, the idempotent generator of $F(\mathcal{I})^*$ in $\mathcal{B}$. Then for every object $x\in B(\mathcal{K})$ we have
    \begin{equation*}
        \overline{F}(x) = \bigoplus\limits_{1\leq i\leq n}(F(k_i)\otimes e_{F(\mathcal{I}_i)^*}).
    \end{equation*}
    Hence, $\overline{F}$ is a well-defined functor such that $\overline{F}\circ\beta=F$, $\overline{F}$ is a Baer functor since $F$ was assumed to be annihilator compatible, and $\overline{F}$ is unique since it is defined on the generators of the objects of $B(\mathcal{K})$.

    Finally, the pair $(B(\mathcal{K}),\beta)$ must be unique since we have constructed a left adjoint to the forgetful functor from Baer tt-categories to essentially small rigid tt-categories. To see this, let $B\colon\mathsf{2CAlg^{rig}}\rightarrow \mathsf{2CBaer}$ denote the previously constructed Baerification functor from the category of essentially small rigid tt-categories and functors between them, to the category of Baer tt-categories and Baer functors, and let $U\colon\mathsf{2CBaer}\rightarrow\mathsf{2CAlg^{rig}}$ denote the forgetful functor from Baer tt-categories to essentially small rigid tt-categories. Then the unit and counit of the adjunction are given by $\eta\colon1_{\mathsf{2CAlg^{rig}}}\rightarrow UB$ and $\epsilon\colon BU\rightarrow 1_{\mathsf{2CBaer}}$ respectively, so 
    \begin{align*}
        &\eta_\mathcal{K} = \beta\colon\mathcal{K}\rightarrow B(\mathcal{K}) ~ \text{ for all } ~ \mathcal{K}\in\mathsf{2CAlg^{rig}},\\ 
        \text{and } &\epsilon_\mathcal{K} = 1_{\mathsf{2CBaer}}\colon\mathcal{K}\rightarrow \mathcal{K} ~ \text{ for all } ~ \mathcal{K}\in\mathsf{2CBaer}. 
    \end{align*}
    To complete the proof, one can see that these natural transformations force the following compositions to be the identity morphisms on $B$ and $U$ respectively:
    \begin{align*}
        &B\xrightarrow{B\eta}BUB\xrightarrow{\epsilon B}B,\\
        &U\xrightarrow{\eta U}UBU\xrightarrow{U\epsilon}U.
    \end{align*}
\end{proof}

The existence of the isomorphism $A(\mathcal{K})\cong\Omega^{\mathrm{reg}}(\mathrm{Spc}(\mathcal{K})^\vee)$ hints at a connection between the construction of Rump and our Baerification construction. The next result will show this connection explicitly. We first require a lemma which gives an important property of $\SpcKinv$.

\begin{lemma}\label{LemmattClosure}
    Let $\mathcal{K}$ be an essentially small tt-category with no nilpotent objects and let $\mathcal{I}$ be a tt-ideal of $\mathcal{K}$. Then the closure of $\supp(\mathcal{I})$ in $\SpcKinv$ is $\overline{\supp(\mathcal{I})}=U(\mathcal{I}^*)$.
\end{lemma}

\begin{proof}
    Let $\mathcal{P}\in\supp(\mathcal{I})$. Then $\mathcal{I}^*\subseteq P$ since $\mathcal{I}\otimes \mathcal{I}^*\simeq0\in \mathcal{P}$. Hence, $\supp(\mathcal{I})\subseteq U(\mathcal{I}^*)$ and therefore $\overline{\supp(\mathcal{I})}\subseteq U(\mathcal{I}^*)$.

    Conversely, let $\mathcal{P}\in U(\mathcal{I}^*)$ and let $N$ be an open neighbourhood of $\mathcal{P}$, so $N$ is of the form $\supp(\mathcal{J})$ for some tt-ideal $\mathcal{J}\subseteq\mathcal{K}$. $\mathcal{J}\nsubseteq\mathcal{P}$ so $\mathcal{J}\nsubseteq\mathcal{I}^*$ and hence, $\mathcal{I}\otimes\mathcal{J}\not\simeq0$. Then, $0\not\simeq \mathcal{I}\otimes\mathcal{J}\subseteq\sqrt{\mathcal{I}\otimes\mathcal{J}}$, therefore,
    \begin{equation*}
        \supp(\mathcal{I})\cap \supp(\mathcal{J})=\supp(\mathcal{I}\otimes\mathcal{J})=\supp(\sqrt{\mathcal{I}\otimes\mathcal{J}})\neq \supp(0)=\varnothing.
    \end{equation*}
    Every neighbourhood of $\mathcal{P}$ thus intersects with $\supp(\mathcal{I})$, i.e., $\mathcal{P}\in\overline{\supp(\mathcal{I})}$ and hence, $U(\mathcal{I}^*)\subseteq\overline{\supp(\mathcal{I})}$ as required.
\end{proof}

\begin{proposition}\label{PropositionBaerRump}
    Let $\K$ be an essentially small rigid tt-category and let $B(\K)$ denote its Baerification given by Construction \ref{ConstructionttBaer}. Then the Hochster dual of the Balmer spectrum of $B(\K)$ is homeomorphic to the absolute of the Hochster dual of the Balmer spectrum of $\K$, given by Construction \ref{ConstructionRump}.
\end{proposition}

\begin{proof}
    The Hochster dual of the Balmer spectrum of $B(\K)$ is given by
    \begin{equation*}
        \mathrm{Spc}(B(\K))^\vee = \mathrm{Spc}\Bigg(\colim_{\substack{\longrightarrow\\D\in\pi(\K)}}\Bigg(\prod\limits_{\D\in D}L_{\supp(\D^*)}\K\Bigg)\Bigg)^\vee.
    \end{equation*}
    By \cite[Proposition 8.2.]{Gal18}, the Balmer spectrum of the filtered colimit of a family of essentially small tt-categories is homeomorphic to the cofiltered limit of the Balmer spectra of those tt-categories. Hence,
    \begin{equation*}
        \mathrm{Spc}(B(\K))^\vee \cong \lim_{\substack{\longleftarrow\\D\in\pi(\K)}} \mathrm{Spc}\Bigg(\prod\limits_{\D\in D}L_{\supp(\D^*)}\K\Bigg)^\vee.
    \end{equation*}
    Now, by \cite[Example 18.16.]{BCHS23}, the Balmer spectrum of a finite product of essentially small tt-categories is homeomorphic to the disjoint union of the Balmer spectra of those tt-categories. Therefore,
    \begin{equation*}
        \mathrm{Spc}(B(\K))^\vee \cong \lim_{\substack{\longleftarrow\\D\in\pi(\K)}} \coprod\limits_{\D\in D} \mathrm{Spc}(L_{\supp(\D^*)}\K)^\vee.
    \end{equation*}
    Finally, $\mathrm{Spc}(L_{\supp(\D^*)}\K)^\vee$ is homeomorphic to $U(\D^*)$ by \cite[Proposition 3.11]{Bal05}, and $U(\D^*)=\overline{\supp(\D)}$ by Lemma \ref{LemmattClosure}. All together we have that
    \begin{equation*}
        \mathrm{Spc}(B(\K))^\vee \cong \lim_{\substack{\longleftarrow\\D\in\pi(\K)}} \coprod\limits_{\D\in D} \overline{\supp(\D)},
    \end{equation*}
    where $\supp(\D)$ is regular open by Lemmas \ref{LemIntRegOpen} and \ref{LemmattClosure}. An application of Proposition \ref{PropositionBooleanIso}, which says that the Boolean algebra of tt-annihilator ideals is isomorphic to the Boolean algebra of regular open subsets of $\SpcKinv$, gives us
    \begin{equation*}
        \mathrm{Spc}(B(\K))^\vee \cong \lim_{\substack{\longleftarrow\\\A\in\Phi(\mathrm{Spc}(\mathcal{K})^\vee)}}~\coprod\limits_{U\in\A}\overline{U},
    \end{equation*}
    precisely the definition of the absolute of $\SpcKinv$ from Construction \ref{ConstructionRump}.
\end{proof}

We conclude by giving an example which shows the Baerification construction in use, and again highlights the connection with the construction of Rump, specifically the computation in Example \ref{ExAbsSpecZinv}.

\begin{example}\label{ExB(DZ)}
    Let $\mathcal{K}=\DZ$. By Lemma \ref{LemAnniffOA}, the tt-annihilator ideals of $\K$ are of the form $\bigcap_{\mathcal{P}\in A}\mathcal{P}$ for $A$ an open set of $\SpcKinv$. Recall from Example \ref{ExDZ} that an open set of $\SpcKinv$ is a union of prime ideals $\P_p=\{k\in\K~|~k\otimes\mathbb{F}_p\simeq0\}$ for $p\in\mathbb{Z}$ non-zero and prime. The tt-annihilator ideals of $\K$ are therefore given by the intersection of any collection of prime tt-ideals $\P_p$ of $\K$.

    We will compute the Baerification of $\K$ in two parts, rationally and over primes. That is, given the functor $\beta\colon\K\rightarrow B(\K)$ of the Baerification construction, we need to compute $\beta(\DZp)$ and $\beta(\DQ)$. Here, $\DZp$ denotes the tt-ideal $\bigcap_{q\neq p}\P_q\cong\mathsf{D}^{\mathrm{perf}}(\mathbb{F}_p)$, which can be thought of as the tt-ideal supported at the prime $p$.

    We first compute $\DZp$. Let $\{\D_i~|~1\leq i\leq n\}$ be some finite partition of tt-annihilator ideals of $\K$. Since the $\D_i$ are pairwise disjoint, there is a unique $\D_j$ such that $\mathbb{F}_p\subseteq\D_j$, and hence $\DZp\subseteq\D_j$. We therefore have that $\DZp\subseteq\D_i^*$ for every $i\neq j$. This gives rise to the following commutative diagram,

\begin{center}
\begin{tikzcd}
\DZp \arrow[rr] \arrow[dd, hook] &  & L_{\supp(\D_j^*)}\K \arrow[dd, hook]            \\
                                 &  &                                                 \\
\K \arrow[rr]                    &  & \prod\limits_{1\leq i\leq n}L_{\supp(\D_i^*)}\K
\end{tikzcd}
\end{center}

which commutes since $\DZp$ is zero in all of the localizations $L_{\supp(\D_i^*)}\K$ except when $i=j$. As a consequence, $\DZp$ embeds fully faithfully into $L_{\supp(\D_j^*)}\K$. Since $\{\D_i~|~1\leq i\leq n\}$ was an arbitrary partition of $\K$, we have that $\DZp$ embeds fully faithfully via $\beta$ into $B(\K)$ as a copy of $\DZp\cong\mathsf{D}^{\mathrm{perf}}(\mathbb{F}_p)$.

Now looking at $\DQ$, since each tt-annihilator ideal $\D$ of $\K$ is an intersection of prime tt-ideals $\P_p$, we have $\D\subseteq\P_0$ for all $\D\in A(\K)$. Therefore, since $\DQ=L_{\supp(\P_0)}\K$, $\DQ$ is contained in every localization of $\K$ by a tt-annihilator ideal. Hence, for a given partition $\{\D_i~|~1\leq i\leq n\}$, $\prod_{1\leq i\leq n}L_{\supp(\D_i^*)}\K$ will contain $n$ copies of $\DQ$. Taking the filtered colimit over all finite partitions we conclude that
\begin{equation*}
    \beta(\DQ) = \prod\limits_\mathbb{N}\DQ = \mathsf{D}^{\mathrm{perf}}\Big(\prod\limits_\mathbb{N}\mathbb{Q}\Big).
\end{equation*}

One can verify that $\mathrm{Spc}(B(\K))^\vee$ is homeomorphic to the absolute of $\SpecZ^\vee$ as computed in Example \ref{ExAbsSpecZinv}, with $\mathsf{D}^{\mathrm{perf}}(\mathbb{F}_p)$ corresponding to the generic points $\P_p$, and $\mathsf{D}^{\mathrm{perf}}\Big(\prod\limits_\mathbb{N}\mathbb{Q}\Big)$ corresponding to the copy of $\beta\mathbb{N}$.
\end{example}

Recall Balmer's Nerves of Steel conjecture, which claims that the Balmer spectrum and the homological spectrum of a given tt-category are always homeomorphic (see \cite{Bal20} for more details). This conjecture has since been disproven in \cite{BHR26}, but it is interesting nonetheless to observe that the above Baerification construction preserves the property of a tt-category satisfying the Nerves of Steel conjecture.

\begin{definition}
    (\cite[Definition 1.1.]{BCHS24}) Let $F\colon\mathcal{C}\rightarrow\mathcal{D}$ be a functor and let $f\colon X\rightarrow Y$ be a morphism in $\mathcal{C}$. Then $F$ is \textbf{nil-conservative} if $F(f)=0$ in $\mathcal{D}$ implies that there is some $n\in\mathbb{N}$ such that $f^{\otimes n}=0$ in $\mathcal{C}$.
\end{definition}

\begin{proposition}
    Let $\mathcal{K}$ be an essentially small rigid tt-category which satisfies the Nerves of Steel conjecture. Then $B(\mathcal{K})$ also satisfies the Nerves of Steel conjecture.
\end{proposition}

\begin{proof}
    By \cite[Corollary 6.15.]{BHSZ26}, if $\mathcal{K}$ satisfies the Nerves of Steel conjecture so to does every finite localization of $\mathcal{K}$. Now consider $\prod_{\mathcal{D}\in D}L_{\supp(\mathcal{D}^*)}\mathcal{K}$ for some partition $D\subseteq A(\mathcal{K})$. The family of projection functors
    \begin{equation*}
        (\pi_\mathcal{E})_{\mathcal{E}\in D}\colon\prod\limits_{\mathcal{D}\in D}L_{\supp(\mathcal{D}^*)}\mathcal{K} \rightarrow L_{\supp(\mathcal{E}^*)}\mathcal{K}
    \end{equation*}
    are jointly nil-conservative, since a morphism being zero in all of the localizations $L_{\supp(\mathcal{D}^*)}\mathcal{K}$ implies that it is zero in the product $\prod_{\mathcal{D}\in D}L_{\supp(\mathcal{D}^*)}\mathcal{K}$. Also, each localization $L_{\supp(\mathcal{D}^*)}\mathcal{K}$ satisfies the Nerves of Steel conjecture, and the induced map
    \begin{equation*}
        \phi\colon\coprod\limits_{\mathcal{D}\in D}\mathrm{Spc}(L_{\supp(\mathcal{D}^*)}\mathcal{K}) \rightarrow \mathrm{Spc}(\prod\limits_{\mathcal{D}\in D}L_{\supp(\mathcal{D}^*)}\mathcal{K})
    \end{equation*}
    is injective (in fact bijective). Hence, by \cite[Proposition 9.2.]{BHSZ26}, $\prod_{\mathcal{D}\in D}L_{\supp(\mathcal{D}^*)}\mathcal{K}$ satisfies the Nerves of Steel conjecture for all partitions $D\subseteq A(\mathcal{K})$. Finally, as a consequence of \cite[Theorem 2.26.]{BBB24} and \cite[Proposition 8.2.]{Gal18}, the filtered colimit of a family of tt-categories which satisfy the Nerves of Steel conjecture will do likewise, which shows that $B(\mathcal{K})$ satisfies the Nerves of Steel conjecture.
\end{proof}

\begin{remark}
    As a consequence of this result, the Baerification of $\DZ$ which was computed in Example \ref{ExB(DZ)} must necessarily satisfy the Nerves of Steel conjecture.
\end{remark}

\printbibliography
\end{document}